\documentclass{amsart}
\usepackage{graphicx}
\usepackage{amssymb}
\usepackage{amsfonts}
\swapnumbers
\newtheorem{theorem}{Theorem}[section]
\newtheorem{lemma}[theorem]{Lemma}
\newtheorem{corollary}[theorem]{Corollary}

\theoremstyle{definition}

\newtheorem{remark}[theorem]{Remark}

\numberwithin{equation}{section}
 \theoremstyle{plain}    
 
 \numberwithin{equation}{section} 
 \numberwithin{figure}{section} 
 \theoremstyle{plain}    
 \theoremstyle{plain}    
 \theoremstyle{remark}    
 \newtheorem*{acknowledgement*}{Acknowledgement} 

\newcommand{\cA}{{\mathcal A}}
\newcommand{\cB}{{\mathcal B}}
\newcommand{\cC}{{\mathcal C}}

\newcommand{\cF}{{\mathcal F}}
\newcommand{\cG}{{\mathcal G}}
\newcommand{\cH}{{\mathcal H}}

\newcommand{\cL}{{\mathcal L}}

\newcommand{\cN}{{\mathcal N}}
\newcommand{\cO}{{\mathcal O}}

\newcommand{\Om}{{\Omega}}
\newcommand{\om}{{\omega}}
\newcommand{\ve}{{\varepsilon}}
\newcommand{\del}{{\delta}}

\newcommand{\gam}{{\gamma}}
\newcommand{\Gam}{{\Gamma}}

\newcommand{\sig}{{\sigma}}
\newcommand{\al}{{\alpha}}
\newcommand{\be}{{\beta}}
\newcommand{\ka}{{\kappa}}
\newcommand{\la}{{\lambda}}

\newcommand{\bbE}{{\mathbb E}}

\newcommand{\bbN}{{\mathbb N}}
\newcommand{\bbP}{{\mathbb P}}

\newcommand{\bbI}{{\mathbb I}}

\begin{document}
\title[]{Poisson and compound Poisson approximations\\ 
in a nonconventional setup}%
 \vskip 0.1cm 
 \author{ Yuri Kifer and Ariel Rapaport\\
\vskip 0.1cm
 Institute  of Mathematics\\
Hebrew University\\
Jerusalem, Israel}%
\address{
Institute of Mathematics, The Hebrew University, Jerusalem 91904, Israel}
\email{ kifer@math.huji.ac.il, ariel.rapaport@mail.huji.ac.il}%

\thanks{ }
\dedicatory{  }
 \date{\today}
\begin{abstract}\noindent
It was shown in \cite{Ki2} that for any subshift of finite type considered
with a Gibbs invariant measure the numbers of multiple recurrencies to 
shrinking cylindrical neighborhoods of almost all points are asymptotically
Poisson distributed. Here we not only extend this result to all $\psi$-mixing
shifts with countable alphabet but actually show that for all points the 
distributions of these numbers are asymptotically close either to Poisson
or to compound Poisson distributions. It turns out that for all nonperiodic
points a limiting distribution is always Poisson while at the same time 
for periodic points there may be no limiting distribution at all unless 
the shift invariant measure is Bernoulli in which case the limiting 
distribution always exists. As a corollary we obtain also that 
the first occurence time of the multiple recurrence event is asymptotically
exponentially distributed. Most of the results are new also
for widely studied single recurrencies (see, for instance, \cite{HV1},
\cite{HV2}, \cite{AS} and \cite{AV}), as well.

\end{abstract}
\maketitle
\markboth{Yu.Kifer and A. Rapaport}{Poisson and compound Poisson approximations} 
\renewcommand{\theequation}{\arabic{section}.\arabic{equation}}
\pagenumbering{arabic}

\section{Introduction}\label{sec1}\setcounter{equation}{0}

Ergodic theorems for nonconventional sums of the form 
\begin{equation}\label{1.1}
S_N=S_N(\om)=\sum_{k=1}^{N}f_1(T^{q_1(k)}\om)f_2(T^{q_2(k)}\om)\cdots
f_\ell(T^{q_\ell (k)}\om),
\end{equation}
were initiated in \cite{Fu1} and employed there in the proof of 
S\' zemeredi's theorem on arithmetic progressions while the name
"nonconventional" comes from \cite{Fu2}. Here $T$ is an ergodic measure
preserving transformation, $f_i$'s are bounded measurable functions and
$q_i(k)=ik$, $i=1,...,\ell$. Since then results concerning such  ergodic
theorems under various conditions evolved into a substantial body of
literature. More recently under appropriate mixing conditions a strong
law   of large numbers and a functional central limit theorem were obtained
even for more general sums in \cite{Ki1} and \cite{KV}, respectively.

If $f_i$ equals for each $i$ an indicator $\bbI_A$ of the same measurable
set $A$ then the corresponding sum $S_N=S_N^A(\om)$ counts the multiple
recurrence events $T^{q_i(k)}\om\in A,\, i=1,...,\ell$ which occur for
$k\leq N$. It was shown in \cite{Ki2} that if we count such multiple
arrivals to appropriately shrinking sets $A_n$ then the sums
\begin{equation}\label{1.2}
S_{N}^{A_n}(\om)=\sum_{k=1}^{N}\bbI_{A_n}(T^{q_1(k)}\om)
\bbI_{A_n}(T^{q_2(k)}\om)\cdots\bbI_{A_n}(T^{q_\ell (k)}\om)
\end{equation}
usually will have asymptotically Poisson distribution for suitably chosen 
sequences $N=N_n\to\infty$ as $n\to\infty$. For $\ell=1$ and $q_1(k)=k$ this
type of results was obtained in a series of papers under various conditions
(see, for instance, \cite{HV1}, \cite{HV2}, \cite{AS} and \cite{AV}).

To explain the results of this paper more precisely let us specify first
our setup which consists of a stationary $\psi$-mixing discrete time
process $\xi(k),\, k=0,1,...$ evolving on a finite or countable state space
 $\cA$ and of
nonnegative increasing functions $q_i,\, i=1,...,\ell$ taking on integer
values on integers and such that $q_1(k)<q_2(k)<\cdots <q_\ell(k)$ when
$k\geq 1$. For each sequence $a=(a_0,a_1,...)\in\cA^\bbN$ of elements
from $\cA$ and any $n\in\bbN$ denote by $a^{(n)}$ the string $a_0,a_1,...
,a_{n-1}$ which determines also an $n$-cylinder set $A^a_n$ in $\cA^\bbN$ 
consisting of sequences whose initial $n$-string coincides with 
$a_0,a_1,...,a_{n-1}$. For appropriately chosen sequences $N=N_n$ we are 
interested in the number of those $l\leq N$ such that the process $\xi(k)=
\xi(k,\om),\, k\geq 0$ repeats the string $a^{(n)}$ starting at times $q_1(l),\,
 q_2(l),...,q_\ell(l)$. Employing the left shift transformation $T$ on the 
sequence space $\cA^\bbN$ we can represent the number in question via 
$S^A_N(\om)$ given by (\ref{1.2}) with $A=A^a_n$ and $N=N_n$ considering
$S^A_N$ as a random variable on the probability space  corresponding to the
process $\xi$.

Viewing such $S^{A^a_n}_N$ as random variables on $\cA^\bbN$ considered with
a Gibbs $T$-invariant measure $\bbP$ it was shown in \cite{Ki2} that 
$S^{A_n^a}_N$ for almost all $a\in\cA^\bbN$ has asymptotically a Poisson 
distribution with a parameter $t$ provided
\begin{equation}\label{1.3}
N=N_n=N_n^a\sim t(\bbP(A^a_n))^{-\ell}\quad\mbox{as}\quad n\to\infty.
\end{equation}
Observe that such asymptotic in $n$ results make sense only for sequences
$a\in\cA^\bbN$ such that $\bbP(A^a_n)>0$ for all $n\geq 1$ but this is
good enough since the set of such $a$'s has probability one.
We will show in this paper for linear $q_i$'s that under (\ref{1.3}) 
for large $n$ the
distribution of $S^{A_n^a}_N$ with $N$ as in (\ref{1.3}) is close  
either to a Poisson distribution with a parameter $t$ or to a compound
Poisson distribution. The latter or the former holds true depending on
whether or not the string $a^{(n)}$ looks periodic with relatively short
with respect to $n$ period. In the "conventional" case $\ell=1$ Poisson
approximation estimates were obtained in \cite{AV0} and \cite{AV} while 
compound Poisson
approximation estimates for periodic sequences were derived in \cite{HV2}
but even in this case the complete dichotomy as described above seems to
be new.

Relying on our approximation estimates it is possible to see (assuming 
(\ref{1.3})) for which sequences $a$ the distribution of $S_N^{A^a_n}$
approaches for large $n$ the Poisson distribution and for which compound
Poisson. Moreover, we will show that under (\ref{1.3}) for all nonperiodic
sequences $a$ the sum $S_N^{A^a_n}$ converges in distribution as $n\to\infty$
to a Poisson random variable with the parameter $t$. On the other hand,
for periodic sequences $a$ the limiting distribution may not exist at all
 and we provide a corresponding example. We observe that
an attempt to construct such an example was made in Section 3.4 of 
\cite{HV2} but it follows both from our estimates and, actually, already
from \cite{AV} that for the example in \cite{HV2} the limiting distribution
exists and it is Poisson. We prove also that if $\bbP$ is a product 
stationary measure on $\cA^\bbN$ (Bernoulli measure), i.e. if coordinate 
projections are independent identically distributed (i.i.d.) random variables,
then the limiting distributions, either Poissonian or compound Poissonian,
exist for all sequences $a$. The above results describe completely the
limiting behavior as $n\to\infty$ in this setup and they seem to be new even
for the extensively studied $\ell=1$ case. Furthermore, either of the above
convergence results ruins also the example in \cite{HV2} since the latter 
counts arrivals to cylinders constructed by a nonperiodic sequence and it is
built on a shift space with a product probability measure. Nevertheless,
as our nonconvergence example shows a slight perturbation of independency 
may already create sequences $a$ where convergence fails.

Let $a\in\cA^\bbN$ and $\tau_{A^a_n}(\om)$ be the first time $k$ when the 
string $a^{(n)}$ starts at each of the places $q_1(k),q_2(k),...,q_\ell(k)$ 
of a sequence $\om\in\cA^\bbN$. When $\ell=1$ and $q_1(k)=k$ such hitting
times were studied in a number of papers (see, for instance, \cite{Ab},
\cite{AS} and references there) where it was shown that under appropriate
normalization they have as $n\to\infty$ exponential limiting distribution.
As a corollary of our Poisson and compound Poisson approximations we will
extend here this type of results to the nonconventional $\ell>1$ situation.

Our results are applicable to larger classes of dynamical systems and not
only to shifts. Indeed, any expansive endomorphism $S$ of a compact $M$, in
particular, any smooth expanding endomorphism of a compact manifolds, has
a symbolic representation as a one sided shift by taking a finite partition
$\al_1,...,\al_k$ of $M$ into sets of small diameter and assigning to each
point $x\in M$ the sequence $j_0,j_1,...$ such that $x\in\cap_{i=0}^\infty
S^{-i}\al_{j_i}$ while noticing that the last intersaction is a singelton
in view of expansivity. Smooth expanding endomorphisms of compact manifolds
have many exponentially fast $\psi$-mixing invariant measures which are
Gibbs measures constructed by H\" older continuous functions which ensures
fast decay of our approximation estimates in this case. Our results remain
valid also for many invertible dynamical systems having symbolic 
representations as two sided shifts, notably, for Axiom A (in particular,
Anosov) diffeomorphisms. Within number theoretic applications the results 
can be formulated in terms of occurence of prescribed strings of digits in
base--$m$ or continued fraction expansions.

The structure of this  paper is the following. In the next section we
describe precisely our setup and conditions, give necessary definitions,
formulate our main results and discuss more their connections and relevance.
In Section \ref{sec3} we state and prove some auxiliary lemmas. In
Sections \ref{sec4} and \ref{sec5} our Poisson and compound Poisson
approximations results will be proved. In Section \ref{sec6} we prove existence
of limiting distributions in the i.i.d. case while
in Section \ref{sec7} we exhibit our nonconvergence example.

\section{Preliminaries and main results}\label{sec2}\setcounter{equation}{0}

We start with a probability space $(\Om,\cF,\bbP)$ such that $\Om$ is
a space of sequences $\cA^\bbN$ with entries from a finite or countable
set (alphabet) $\cA$ which is not a singelton, the $\sigma$-algebra $\cF$ 
is  generated by cylinder sets and a
probability $\bbP$ invariant with respect to the left shift $T$ acting by
$(T\om)_i=\om_{i+1}$ for $\om=(\om_0,\om_1,...)\in\Om=\cA^{\bbN}$. For each
$I\subset\bbN$ denote by $\cF_I$ the sub $\sig$-algebra of $\cF$ generated
by the cylinder sets $[a_i,\, i\in I]=\{\om=(\om_0,\om_1,...)\in\Om:\,
\om_i=a_i\,\forall i\in I\}$. Without loss of generality we assume that
the probability of each 1-cylinder set is positive, i.e. $\bbP([a])>0$ for 
every $a\in\cA$, and since $\cA$ is not a singelton we have also 
$\sup_{a\in\cA}\bbP([a])<1$. Our results will be based on the $\psi$-mixing
(dependence) coefficient defined for any two $\sig$-algebras $\cG,\cH\subset
\cF$ by (see \cite{Br}),
\begin{eqnarray}\label{2.1}
&\psi(\cG,\cH)=\sup_{B\in\cG,\, C\in\cH}\big\{\big\vert\frac {\bbP(B\cap C)}
{\bbP(B)\bbP(C)}-1\big\vert,\,\,\bbP(B)\bbP(C)\ne 0\big\}\\
&=\sup\{\| \bbE(g|\cG)-Eg\|_\infty :\, g\,\,\,\mbox{is}\,\,\,
\cH-\mbox{measurable and}\,\,\, \bbE|g|\leq 1\}\nonumber
\end{eqnarray}
where $\|\cdot\|_p$ is the $L^p(\Om,\bbP)$-norm. Next, we set
\begin{equation}\label{2.2}
\psi_m=\sup_{n\in\bbN}\psi(\cF_{0,n},\cF_{n+m+1,\infty})
\end{equation}
where $\cF_{k,n}=\cF_{\{i:\,k\leq i\leq n\}}$ when $0\leq k\leq n\leq\infty$.
It follows from (\ref{2.1}) and (\ref{2.2}) that $\psi_m$ is non increasing
in $m$ and the measure $\bbP$ is called $\psi$-mixing if $\psi_0<\infty$  
and $\psi_m\to 0$ as $m\to\infty$.

Another important ingredient of our setup is a collection of $\ell$
positive increasing integer valued functions $q_1,...,q_\ell$ defined
on positive integers $\bbN_+$ and such that
\begin{eqnarray}\label{2.3}
&q_1(n)<q_2(n)<\cdots <q_{\ell}(n),\,\,\forall n\in\bbN_+\,\,\mbox{and}\\
&\lim_{n\to\infty}(q_{i+1}(n)-q_i(n))=\infty,\,\,\forall i=1,...,\ell -1.
\nonumber\end{eqnarray}
Our estimates will involve the following quantities related to these functions
\begin{equation}\label{2.4}
g(n)=\inf_{k\geq n}\min_{1\leq i\leq\ell -1}(q_{i+1}(k)-q_i(k))\,\,\mbox{and}
\,\,\gam(n)=\min\{ k\geq 0:\, g(k)\geq 2n\}
\end{equation}
where $n\in\bbN_+$. In view of the second assumption in (\ref{2.3}) the function
$\gam(n)$ is well defined for all $n$.

Denote by  $\cC_n$ the set of all $n$-cylinders $A=[a_0,,a_1,...,
a_{n-1}]$. Here and in what follows we denote as above by $[Q]$ a cylinder if 
$Q$ is a finite string of elements from $\cA$ but, as usual, when $Q$ is a
 number then $[Q]$ denotes the integral part of $Q$.
Similarly to \cite{AV} we introduce also for each $A=[a_0,a_1,...,a_{n-1}]
\in\cC_n$ the quantity
\begin{equation}\label{2.5}
\pi(A)=\min\{ k\in\{ 1,2,...,n\}:\, A\cap T^{-k}A\ne\emptyset\}
\end{equation}
setting also $A(\pi)=[a_0,a_1,...,a_{\pi(A)-1}]\in\cC_{\pi(A)}$.
Next, for each $A\in\cC_n$ we write
\begin{equation}\label{2.6}
S^A_N=\sum_{k=1}^NX^A_k\,\,\mbox{where}\,\, X^A_k=\prod_{i=1}^\ell
\bbI_A\circ T^{q_i(k)}
\end{equation}
and often, when $A$ will be fixed and clear, we will drop the index $A$
in $X_k^A$ and $S^A_N$. In order to shorten our formulas we will use
throughout this paper the notation $\wp(x)=xe^x$.
The following result exhibits our Poisson approximation estimates
(for $\ell=1$ a  similar estimate was obtained in \cite{AV}).

\begin{theorem}\label{thm2.1}
Let $A\in\cC_n$ with $\bbP(A)>0$, $t>0$, $N=[t(\bbP(A))^{-\ell}]$ and 
assume that $\psi_n<(3/2)^{1/(\ell+1)}-1$. Then
\begin{eqnarray}\label{2.7}
&\sup_{L\subset\bbN}|\bbP\{ S^A_N\in L\}-P_t(L)|\leq 16\bbP(A)\big(\ell^2nt+
\gam(n)(1+t^{-1}))\\
&+6\bbP(A(\pi))tn\ell^2(1+\psi_0)+2\wp(2^\ell t\psi_n+\gam(n)\bbP(A))
\nonumber\end{eqnarray}
where $P_t(L)$ is the probability assigned to $L$ by the Poisson distribution
with the parameter $t$. Moreover, assume that $\bbP$ is $\psi$-mixing (i.e.
$\psi_n\to 0$ as $n\to\infty$) then (\ref{2.7}) holds true with
\begin{equation}\label{2.8}
\bbP(A)\leq e^{-\Gam n}\,\,\mbox{and}\,\,\bbP(A(\pi))\leq
e^{-\Gam\pi(A)}
\end{equation}
for some $\Gam>0$ independent of $A$ and $n$.
\end{theorem}

Next, for each $\om=(\om_0,\om_1,...)\in\Om$ and $n\geq 1$ set $A^\om_n=
[\om_0,.....,\om_{n-1}]$. Denote by $\Om_\bbP$ the set of $\om\in\Om$
such that $\bbP(A^\om_n)>0$ for all $n\geq 1$. Clearly, $T\Om_\bbP\subset
\Om_\bbP$ and $\bbP(\Om_\bbP)=1$ since $\Om_\bbP$ is the complement in $\Om$
of the union of all cylinder sets $A\in\cC_n,\, n\geq 1$ such that $\bbP(A)=0$
and the number of such cylinders is countable. Theorem \ref{thm2.1} yields the 
following asymptotic result.

\begin{corollary}\label{cor2.2} Assume that the conditions of Theorem 
\ref{thm2.1} together with the $\psi$-mixing assumption hold true. Then
for $\bbP$-almost all $\om\in\Om$ there exists $M(\om)<\infty$ such that for
any $n\geq M(\om)$,
\begin{eqnarray}\label{2.9}
&\sup_{L\subset\bbN} |\bbP\{ S^{A^\om_n}_N\in L\}-P_t(L)|\leq 16e^{-\Gam n}
\big(\ell^2nt+\gam(n)(1+t^{-1})\\
&+t\ell^2 n^4(1+\psi_0)\big)+2\wp(2^\ell t\psi_n+\gam(n)e^{-\Gam n})\nonumber
\end{eqnarray}
provided $N=N^\om_n=[t(\bbP(A^\om_n))^{-\ell}]$. 
 \end{corollary}

 Theorem \ref{thm2.1} yields that if both $\pi(A^\om_n)$ and $g(n)$ grow
 fast enough in $n$ then under the $\psi$-mixing condition the distribution
 of $S^A_N$ approaches the Poisson one as $n\to\infty$. The fast growth
 of $\pi(A^\om_n)$ in $n$ means that $\om$ is "very" nonperiodic and
 Corollary \ref{cor2.2} sais that almost all $\om$'s fall into this category.
 The above results can be compared with Theorem 2.3 from \cite{Ki2} where
 $\bbP$ was supposed to be a Gibbs invariant measure corresponding to a
 H\" older continuous function concentrated on a subshift of finite type
 space (see \cite{Bo}). A $\psi$-mixing coefficient of such a measure
 decays in $n$ exponentially fast (see \cite{Bo}), and so (\ref{2.8}) holds
 true then. Under conditions of \cite{Ki2} the function $g(n)$ grows faster
 than logarithmically, and so $\gam(n)$ grows slower than exponentially which
 yields a fast decay of the errors in the Poisson approximations above
 which improves the result from \cite{Ki2} where no estimates of the
 convergence rate were obtained.
 Note also that any Gibbs measure $\bbP$ gives a positive weight to each
 cylinder in the corresponding subshift of finite type space, and so in
 this case the latter space coincides with $\Om_\bbP$.
 We observe that in the "conventional" $\ell=1$ case similar to Theorem 
 \ref{thm2.1} error estimates in the corresponding Poisson approximations
 were obtained in \cite{AV0} and \cite{AV}.

 Next, we describe our compound Poisson approximations where we assume that
 \begin{equation}\label{2.10}
 q_i(k)=d_ik\,\,\mbox{for some}\,\, d_i\in\bbN,\, i=1,....,\ell\,\,\mbox{with}
 \,\, 1\leq d_1<d_2<\cdots <d_\ell.
 \end{equation}
 For each $R=[a_0,a_1,.....,a_{r-1}]\in\cC_r$ and an integer $n\geq r$ set
 \begin{equation*}
 R^{n/r}=\big(\cap_{k=0}^{[n/r]-1}T^{-kr}(R)\big)\cap T^{-[n/r]}\big([a_0,...,
 a_{n-r[n/r]-1}])
 \end{equation*}
 where we define $[a_0,...,a_{n-r[n/r]-1}]=\Om$ if $r$ divides $n$. Observe
 that if $A=[a_0,a_1,...,a_{n-1}]$, $A(\pi)=R$ and $\pi(A)=r$ then $R^{n/r}=A$.

 \begin{theorem}\label{thm2.3}
 Let $A=[a_0,a_1,...,a_{n-1}]\in\cC_n$ with $\bbP(A)>0$, $t>0$,
 $N=[t(\bbP(A))^{-\ell}]$ and assume that $\psi_n<(3/2)^{1/(\ell+1)}-1$. Set 
 $r=\pi(A),\,\, R=[a_0,.....,a_{r-1}]$,
 \begin{equation}\label{2.11}
 \ka=\mbox{lcm}\big\{\frac {r}{\mbox{gcd}\{r,d_i\}}:\, 1\leq i\leq\ell\big\}
 \,\,\mbox{and}\,\,\rho=\rho_A=\prod_{i=1}^\ell\bbP\{ R^{(n+d_i\ka)/r}|A\}
 \end{equation}
 where lcm and gcd denote the least common multiple and the greatest common
 divisor, respectively. Assume that $\bbP$ is $\psi$-mixing then
 \begin{equation}\label{2.12}
 \sup_{A\in\cC_n,\, n\geq 1}\rho_A<1
 \end{equation}
 and if $\om\in\Om$ is not a periodic sequence then
 \begin{equation}\label{2.13}
 \lim_{n\to\infty}\rho_{A^\om_n}=0
 \end{equation}
 where $A_n^\om$ is the same as in Corollary \ref{cor2.2}. Furthermore, 
 let $W$ be a Poisson random variable with the parameter $t(1-\rho)$.
 Then for any $n>r(d_\ell +6)$ there exists a sequence of i.i.d. random
 variables  $\eta_1,\eta_2,...$ independent of $W$ such that
 $\bbP\{\eta_1\in\{ 1,...,[\frac nr]\}\}=1$ and the compound Poisson random
 variable $Z=\sum_{k=1}^W\eta_k$ satisfies
 \begin{eqnarray}\label{2.14}
 &\sup_{L\subset\bbN}|\bbP\{S^A_N\in L\}-\bbP\{ Z\in L\}|\leq 2^{2\ell+7}(1+
 \psi_0)^{2\ell}\big(d_\ell\ell^2n^4e^{-\Gam n/2}\\
 &+\psi_n(1-e^{-\Gam})^{-1}\big)+2\wp\big(10(1+\psi_0)^{2\ell}d_\ell n^2(t+1)
 e^{-\Gam n/2}+2^\ell t\psi_n\big).
 \nonumber\end{eqnarray}
 \end{theorem}
 
 We will see in Lemma \ref{lem3.7} that for any nonperiodic sequence 
 $\om\in\Om$ both (\ref{2.13}) and $r^\om_n=\pi(A^\om_n)\to\infty$ as
 $n\to\infty$ hold true which together with appropriate estimates for
 the distribution of $\eta_i$'s constructed in Theorem \ref{thm2.3} will
 yield the following result.
 
 \begin{corollary}\label{cor2.3+} Under conditions and notations of
 Theorem \ref{thm2.3} for any nonperiodic sequence $\om\in\Om_\bbP$ the sum
  $S_N^{A_n^\om}$ converges in distribution to a Poisson random variable
  with the parameter $t$.
  \end{corollary}
 
 An important role in the proof of estimates in both Theorems \ref{thm2.1}
 and \ref{thm2.3} will play the estimates from \cite{AGG}. We observe that
 both Theorem \ref{thm2.3} and Corollary \ref{cor2.3+} seem to be new  even
 for the "conventional" case $\ell=1$. Some compound Poisson approximation
 estimates were obtained in \cite{HV2} but they deal there only with cylinders
 $A=A^a_n$ constructed for periodic sequences $a$ and only with geometrically
 distributed random variables $\eta_i$ appearing in the compound Poisson 
 random variable $Z$. Only results about convergence for almost all sequences
 $\om$ were known before while Corollary \ref{cor2.3+} asserts convergence
 for all but periodic points whose number is countable.

 \begin{remark}\label{rem2.4}
 In fact, all above results can be formulated assuming that $N\sim
 t(\bbP(A))^{-\ell}$ instead of $N=[t(\bbP(A))^{-\ell}]$ where $A\in\cC_n$.
 Indeed, in view of
 Lemma \ref{lem3.2} of the next section for all sufficiently large
 (in comparison to $n$) $k$,
 \[
 \bbE X^A_k\leq C(P(A))^\ell
 \]
 for some $C>0$, and so
 \[
 \sum_{\min(N,t(\bbP(A))^{-\ell})\leq k\leq\max(N,t(\bbP(A))^{-\ell)})}
 X^A_k\to 0\,\,\mbox{in probability as}\,\, n\to\infty.
 \]
 \end{remark}

 Next, for any $A\in\cC_n$ define
 \begin{equation}\label{2.15}
 \tau_A(\om)=\min\{ k\geq 1:\,\prod_{i=1}^\ell\bbI_A\circ T^{q_i(k)}(\om)=1\}
 \end{equation}
 which is the first time $k$ for which the multiple recurrence event
 $\{ T^{q_i(k)}\om\in A,\, i=1,...,\ell\}$ occurs. Then combining
 Theorems \ref{thm2.1} and \ref{thm2.3} (assuming zero occurence of
 multi recurrence events in question) we derive

 \begin{corollary}\label{cor2.5}
 Suppose that $\bbP$ is $\psi$-mixing and $q_i(k)=d_ik,\, i=1,...,\ell$.
 Then for any $A\in\cC_n$ with $\bbP(A)>0$ and $t\geq 0$,
 \begin{eqnarray}\label{2.16}
 &\quad|\bbP\{ (\bbP(A))^\ell\tau_A>t\}-e^{-(1-\rho_A)t}|\leq 2^{2\ell+8}(1+
 \psi_0)^{2\ell}(t+1)\big(\psi_n(1-e^{-\Gam})^{-1}\\
 &+d_\ell\ell^2n^4(1+t^{-1})e^{-\Gam n/(d_\ell+6)}\big)
 +2\wp\big(2^\ell t\psi_n+10e^{-\Gam n/2}(1+\psi_0)^{2\ell}d_\ell n^2(t+1)\big),
 \nonumber\end{eqnarray}
 provided $\psi_n<(3/2)^{1/(\ell+1)}-1$.
 \end{corollary}

 We observe that when $\ell=1$ the random variable $\tau_A$ can be treated
 as a stopping time and most direct methods which deal with $\tau_A$ in this
 case are based on this fact (see, for instance, \cite{Ab}, \cite{AS} and
 references there). In our nonconventional situation when $\ell>1$ the
 random variable $\tau_A$ depends on the future and it is difficult to
 deal with it directly. By this reason we obtain Corollary \ref{cor2.5} as
 an immediate consequence of Poisson and compound Poisson approximations
 applied to the case when no multiple recurrence event occurs until
 time $t(\bbP(A))^{-\ell}$. Observe that if we replace $t$ in (\ref{2.16}) by
 $u=t(1-\rho_A)$
 then in view of (\ref{2.12}) Corollary \ref{cor2.5} will take the form of
 Theorem 1 from \cite{Ab} which dealt though only with the "conventional" 
 $\ell=1$ case. Recall, also that arrivals to more general than cylindrical
 sets may lead to other limiting distribution (see \cite{La}) so in order
 to remain in the class of Poisson and compound Poisson distributions
 we have to restrict the investigation to arrivals to cylinder sets.
 
 In Section \ref{sec5} we will derive from Corollary \ref{cor2.5} together
 with the Shannon-McMillan-Breiman theorem the following result.
 
 \begin{corollary}\label{cor2.5+}
 Suppose that (\ref{2.10}) holds true and the $\psi$-mixing coefficient 
 of $\bbP$ satisfies 
 \begin{equation}\label{2.101}
 \sum_{n=1}^\infty\psi_n\ln n<\infty.
 \end{equation}
 Then for $\bbP$-almost all $\om\in\Om$,
 \begin{equation}\label{2.102}
 \lim_{n\to\infty}\frac 1n\big(\ln\tau_{A^\om_n}+\ell\ln\bbP(A^\om_n)\big)=0
 \,\,\,\,\bbP-\mbox{a.s.}
 \end{equation}
 If, in addition,
 \begin{equation}\label{2.103}
 -\sum_{a\in\cA}\bbP([a])\ln\bbP([a])<\infty
 \end{equation}
 then for $\bbP$-almost all $\om\in\Om$,
 \begin{equation}\label{2.104}
 \lim_{n\to\infty}\frac 1n\ln\tau_{A^\om_n}=h_\bbP(T)\,\,\,\,\bbP-\mbox{a.s.}
 \end{equation}
 where $h_\bbP(T)$ is the Kolmogorov-Sinai entropy of the shift $T$ with
 respect to its invariant measure $\bbP$.
 \end{corollary}
 In the "conventional" $\ell=1$ case relying on estimates of Theorem 1 from
 \cite{Ab} it is possible to obtain (\ref{2.104}) under weaker than 
 (\ref{2.101}) conditions.
 
 Under additional independency conditions we obtain a more specific 
 than in Theorem \ref{thm2.3}
 compound Poisson approximation of $S^A_n$ constructed via independent 
 geometrically distributed random variables.

 \begin{theorem}\label{thm2.6}
 Suppose that the conditions of Theorem \ref{thm2.3} are satisfied and  assume
 in addition that $\bbP$ is the product stationary probability on $\cA^\bbN$,
 i.e. that coordinate projections from $\Om$ to $\cA$ are i.i.d. random
 variables. Let $W$ be the same as in Theorem \ref{thm2.3} and $\zeta_1,
 \zeta_2,...$ be i.i.d. random variables independent of $W$ and having the
 geometric distribution with the parameter $\rho$ defined in Theorem
 \ref{thm2.3}, i.e. $\bbP\{\zeta_1=k\}=(1-\rho)\rho^{k-1},\, k=1,2,...$.
 Let $Y=\sum_{k=1}^W\zeta_k$ then
 \begin{eqnarray}\label{2.17}
 &\sup_{L\subset\bbN}|\bbP\{S^A_N\in L\}-\bbP\{ Y\in L\}|\\
 &\leq 2^{2\ell+8}(t+1)\ell^2d_\ell n^4e^{-\Gam n/2}+2\wp(12d_\ell n^2(t+1)
 e^{-\Gam n/2})+\frac {t^{[n/r]+1}}{([n/r]+1)!}.
 \nonumber\end{eqnarray}
 where $A,\, N,$ and $r$ are the same as in Theorem \ref{thm2.3}.
 \end{theorem}

We observe that when $\bbP$ is a product stationary measure on $\Om=\cA^\bbN$
then $\Om_\bbP=\Om$ since then in view of our assumption that $\bbP([a])>0$
for all $a\in\cA$ any cylinder set has positive probability.
The following result,
 which seems to be new even in the "conventional" $\ell=1$ case, sais
 that under the independency conditions of Theorem \ref{thm2.6} the limiting
 distribution, either Poisson or compound Poisson, always exists.

 \begin{theorem}\label{thm2.7}
 Suppose that the conditions of Theorem \ref{thm2.6} are satisfied.
  For each $\om\in\Om$ and $n\geq 1$ set $U^\om_n=S^{A^\om_n}_{N^\om_n}$
 where $A^\om_n$ is the   same as in Corollary \ref{cor2.2},
 $N^\om_n=[t(\bbP(A^\om_n))^{-\ell}]$ and $t>0$. Let $r^\om_n=\pi(A^\om_n)$,
 $R^\om_n=[\om_0,\om_1,...,\om_{r^\om_n-1}]$ and $\rho^\om_n=\bbP(R^\om_n)$.
 Then the limits
 \begin{equation}\label{2.18}
 \lim_{n\to\infty}r^\om_n=r^\om\,\,\mbox{and}\,\,\lim_{n\to\infty}
 \rho_n^\om=\rho^\om
 \end{equation}
 exist for any $\om\in\Om$.
 
 (i) If $\om$ is not a periodic point then $U^\om_n$ converges in
 distribution as $n\to\infty$ to a Poisson random variable with the
 parameter $t$.

 (ii) If $\om$ is a periodic point (sequence) of a minimal period $d$ then
 $r^\om=d$ and $\rho^\om=(\bbP(A^\om_d))^{k_0}$ where
 \[
 k_0=\frac {\ka^\om}{r^\om}\sum_{i=1}^\ell d_i
 \]
 with $\ka^\om$ defined by (\ref{2.11}) for $r^\om$ in place of $r$ there.
 Furthermore, let $W$ be a Poisson
 random variable with the parameter $t(1-\rho^\om)$ and $\zeta_1,\zeta_2,
 ...$ be a sequence of i.i.d. random variables independent of $W$ having
 the geometric distribution with the parameter $\rho^\om$, i.e.
 $\bbP\{\zeta_1=k\}=(1-\rho^\om)(\rho^\om)^{k-1},\, k=1,2,...$. Then
 $U^\om_n$ converges in  distribution to $Y^\om=\sum_{k=1}^W\zeta_k$ as
 $n\to\infty$.
 \end{theorem}
 
 A number theory (combinatorial) application of Theorem \ref{thm2.7}
 can be described in the following way. Fix a point $a\in[0,1)$ having
 a base $m\geq 2$ expansion $a=\sum_{i=0}^\infty a_im^{-(i+1)}$. For
 each point $\om\in[0,1)$ with a base $m$ expansion $\om=\sum_{i=0}^\infty
  \om_im^{-(i+1)}$ count the number $N_n(\om)$ of those $i\leq tm^n$ for
  which the initial $n$-string $a_0,a_1,...,a_{n-1}$ of $a$ is repeated 
  starting from all places $d_1i,d_2i,...,d_\ell i$ in the expansion of $\om$.
   Then the random
  variable $N_n(\om)$ on the probability space $([0,1],$Leb) converges 
  in distribution either to a Poisson (with the parameter $t$) or to a
  compound Poisson random variable depending on whether the expansion
  of $a$ is nonperiodic or periodic. In particular, for all irrational
  and for some rational $a$ the limiting distribution will be Poissonian.
  
  Since we allow also countable alphabet shift spaces our results are 
  applicable to continued fraction expansions, as well. There the Gauss
  map $Gx=\frac 1x$ (mod 1), $x\in(0,1],\, G0=0$ preserves the Gauss
  measure $\mu_G(\Gam)=\frac 1{\ln 2}\int_\Gam\frac {dx}{1+x}$ which
  is exponentially fast $\psi$-mixing (see, for instance, \cite{He}) and
  $G$ (restricted to points with infinite continued fraction expansions)
  is conjugate to the left shift on a sequence space with a countable 
  alphabet. Thus Theorems \ref{thm2.1} and \ref{thm2.3} are applicable
  here, as well. On the other hand, Theorem \ref{thm2.7} does not work for 
  the Gauss measure which does not produce i.i.d. digits in continued
  fraction expansions (see, for instance, \cite{KPW}), and so in order
  to apply Theorem \ref{thm2.7} here we have to take other probabability
  measures conjugated to product measures on the shift space which produce
  i.i.d. continued fraction digits. Moreover, these results remain valid
  for a larger class of transformations generated by, so called,
  $f$-expansions (see \cite{He}).

  It follows from both Corollary \ref{cor2.3+} and Theorem \ref{thm2.7}
   that the nonconvergence example from
  Section 3.4 in \cite{HV2} is not correct and, in fact, the limiting
  distribution exists  there and it is Poisson since the cylinder sets
  constructed there were based on a nonperiodic point and the probability
  measure on the sequence space was a product measure, whence both
  Corollary \ref{cor2.3+} and Theorem
  \ref{thm2.7} above are applicable. Still, we construct in Section \ref{sec7}
  an example of a point $\om\in\Om$ and of a $\psi$-mixing shift invariant
  measure such that $S^{A^\om_n}_{N^\om_n}$ considered with
  $\ell=1$ and $d_1=1$ does not have a limiting
  distribution as $n\to\infty$. Moreover, the $\psi$-mixing coefficient 
  $\psi_l$ in our example equals zero for any $l\geq 1$ while $\psi_0<\infty$,
  i.e. the situation there is as close as possible to independence where
  Theorem \ref{thm2.7} asserts convergence for all points $\om\in\Om$.
  
 \begin{remark}\label{rem2.8} 
 Observe that our convergence theorems above are based on approximation
 estimates of Theorems \ref{thm2.1} and \ref{thm2.3} which yield convergence
 of distributions with respect to the total variation distance but for
 integer valued random variables we are dealing with this is equivalent
 to convergence in distribution.
  \end{remark}
  \begin{remark}\label{rem2.9}
All above results were stated for one-sided shifts but essentially without
changes they remain true for two-sided shifts, as well. We can also
restrict the discussion to a subshift of finite type space considering
invariant probability measures supported there. This enables us to apply
the results to Axiom A diffeomorphisms (see \cite{Bo}) considered with
their Gibbs invariant measures relying on their symbolic representations
via Markov partitions.
\end{remark}

\section{Auxiliary lemmas}\label{sec3}\setcounter{equation}{0}

We start with the following result.
\begin{lemma}\label{lem3.1}
Suppose that $\bbP$ is $\psi$-mixing then there exists a constant $\Gam>0$
such that for any $A\in\cC_n$,
\begin{equation}\label{3.1}
\bbP(A)\leq e^{-\Gam n}.
\end{equation}
\end{lemma}
\begin{proof} The proof is contained in Lemma 1 from \cite{Ab0} and in
Lemma 1 from \cite{GS}. In both places the authors assume summability
of $\psi_n$ but, in fact, both proofs use only convergence of $\psi_n$
to zero as $n\to\infty$ which is $\psi$-mixing. In both papers the result
is formulated with a constant in front of the exponent in (\ref{3.1}) but
it is clear from the proofs there that this constant can be taken equal 1.
Furthermore, the authors there work on finite alphabet shift spaces but the
proof of this result remains valid without any changes for countable alphabets,
as well.
\end{proof}

We will need also the following result which is, essentially, well known
but for readers' convenience we give its simple proof here.
\begin{lemma}\label{lem3.2}
Let $Q_1,....,Q_l,\, l\geq 1$ be subsets of nonnegative integers such that 
for $i=2,...,l-1$,
\begin{equation}\label{3.2}
\max Q_{i-1}+k\leq\min Q_i\leq\max Q_i\leq\min Q_{i+1}-k
\end{equation}
for some integer $k>0$. Then for any $U_i\in\cF_{Q_i},\, i=1,...,l$,
\begin{equation}\label{3.3}
|\bbP(\cap_{i=1}^lU_i)-\prod_{i=1}^l\bbP(U_i)|\leq ((1+\psi_k)^l-1)
\prod_{i=1}^l \bbP(U_i).
\end{equation}
\end{lemma}
\begin{proof} By the definition (\ref{2.1}) and (\ref{2.2}) of the coefficient
 $\psi$,
\begin{equation}\label{3.4} 
|\bbP(\cap_{i=1}^{j+1}U_i)-\bbP(\cap_{i=1}^jU_i)\bbP(U_{j+1})|\leq\psi_k
\bbP(\cap_{i=1}^jU_i)\bbP(U_{j+1})
\end{equation}
for $j=1,...,l-1$. applying (\ref{3.4}) successively $l-1$ times we obtain
\begin{equation}\label{3.5}
|\bbP(\cap_{i=1}^lU_i)-\prod_{i=1}^l\bbP(U_i)|\leq\psi_k\sum_{j=1}^{l-1}
\big( \bbP(\cap_{i=1}^jU_i)\prod_{i=j+1}^l\bbP(U_i)\big).
\end{equation}
Furthermore, applying (\ref{3.4}) successively $j-1$ times we see that
\begin{equation}\label{3.6}
\bbP(\cap_{i=1}^jU_i)\leq (1+\psi_k)\bbP(\cap_{i=1}^{j-1}U_i)\bbP(U_j)\leq
(1+\psi_k)^{j-1}\prod_{i=1}^j\bbP(U_i).
\end{equation}
This together with (\ref{3.5}) yields (\ref{3.3}).
\end{proof}

\begin{lemma}\label{lem3.3}
let $Q$ and $\tilde Q$ be two subsets of nonnegative integers such that at 
least one of these sets is finite and 
\[
d=\min_{i\in Q,\, j\in R}|i-j|>0.
\]
Let $k$ bound the number of components of $Q$ and $\tilde Q$ which are 
separated
by some elements of the other set. Assume that $\psi_d<2^{1/k}-1$. Then
\begin{equation}\label{3.7}
\psi(\cF_Q,\cF_{\tilde Q})\leq 2^{2k+2}\psi_d(2-(1+\psi_d)^k)^{-2}.
\end{equation}
\end{lemma}
\begin{proof}
Suppose, for instance, that $Q$ is a finite set. Then $Q$ and $\tilde Q$ can be
represented as disjoint unions
\[
Q=\cup_{i=1}^kQ_{2i-1}\,\,\mbox{and}\,\, \tilde Q=\cup_{i=1}^kQ_{2i}
\]
such that $Q_1$ and $Q_{2k}$ may be empty sets while all $Q_j,\, j=2,...,2k-1$
are nonempty and
\begin{equation}\label{3.8}
\max Q_{j-1}+d\leq\min Q_j\leq\max Q_j\leq\min Q_{j+1}-d
\end{equation}
for $j=2,...,2k-1$ where if $j=2$ and $Q_1=\emptyset$ or if $j=2k-1$ and 
$Q_{2k}=\emptyset$ then we disregard the first or the last inequality in
(\ref{3.8}), respectively.

Next, let $U=\cap_{i=1}^kU_{2i-1}$ and $V=\cap_{i=1}^kU_{2i}$ where
$U_{2i-1}\in\cF_{Q_{2i-1}}$ and $U_{2i}\in\cF_{Q_{2i}},\, i=1,...,k$.
Then by Lemma \ref{lem3.2},
\begin{equation}\label{3.9}
|\bbP(U\cap V)-\prod_{j=1}^{2k}\bbP(U_j)|\leq\big((1+\psi_d)^{2k}-1\big)
\prod_{j=1}^{2k}\bbP(U_j),
\end{equation}
\begin{equation}\label{3.10}
|\bbP(U)-\prod_{i=1}^{k}\bbP(U_{2i-1})|\leq\big((1+\psi_d)^{k}-1\big)
\prod_{i=1}^{k}\bbP(U_{2i-1})
\end{equation}
and
\begin{equation}\label{3.11}
|\bbP(V)-\prod_{i=1}^{k}\bbP(U_{2i})|\leq\big((1+\psi_d)^{k}-1\big)
\prod_{i=1}^{k}\bbP(U_{2i}).
\end{equation}
Combining (\ref{3.9})--(\ref{3.11}) we obtain that
\begin{equation}\label{3.12}
|\bbP(U\cap V)-\bbP(U)\bbP(V)|\leq\big((1+\psi_d)^{2k}-1\big)\prod_{i=1}^{k}
\bbP(U_{2i-1})\big(\bbP(V)+2\prod_{i=1}^{k}\bbP(U_{2i})\big).
\end{equation}
Next, by Lemma \ref{lem3.2},
\begin{equation}\label{3.13}
\bbP(U)\geq (2-(1+\psi_d)^k)\prod_{i=1}^{k}\bbP(U_{2i-1})\,\,\mbox{and}\,\,
\bbP(V)\geq (2-(1+\psi_d)^k)\prod_{i=1}^{k}\bbP(U_{2i}),
\end{equation}
which together with (\ref{3.12}) yields that
\begin{eqnarray}\label{3.14}
&|\bbP(U\cap V)-\bbP(U)\bbP(V)|\\
&\leq ((1+\psi_d)^{2k}-1)(4-(1+\psi_d)^k)(2-(1+\psi_d)^k)^{-2}\bbP(U)\bbP(V).
\nonumber\end{eqnarray}
The inequality (\ref{3.14}) remains true when $U$ and $V$ are disjoint unions
of intersections of sets as above and it will still holds under monotone
limits of sets $U\in\cF_Q$ and $V\in\cF_{\tilde Q}$. Hence, it holds true for 
all $U\in\cF_Q$ and $V\in\cF_{\tilde Q}$, and so (\ref{3.7}) follows 
(see \cite{Bi}).
\end{proof}

We will need also the following estimate of the total variation distance
between two Poisson distributions.
\begin{lemma}\label{lem3.4} For any $\lambda,\gamma>0$, 
\begin{equation}\label{3.15}
\sum_{l=0}^\infty |P_{\lambda}(l)-P_{\gamma}(l)|
\leq2e^{|\lambda-\gamma|}|\lambda-\gamma|=2\wp(|\la-\gam|).
\end{equation}
\end{lemma}
\begin{proof} Assume, for instance, that $\la\geq\gam$. Then
\begin{eqnarray*}
&\sum_{n=0}^\infty|e^{-\la}\frac {\la^n}{n!}-e^{-\gam}\frac {\gam^n}{n!}|
\leq\sum_{n=0}^\infty e^{-\la}\frac {\la^n}{n!}\big(|1-e^{\la-\gam}|\\
&+e^{\la-\gam}|1-(\frac \gam\la)^n|\big)\leq |1-e^{\la-\gam}|+
e^{\la-\gam}|\la-\gam|\leq 2(\la-\gam)e^{\la-\gam}
\end{eqnarray*}
and (\ref{3.15}) follows.
\end{proof}

The following two lemmas will be used in the proof of Theorem \ref{thm2.3}.

\begin{lemma}\label{lem3.5} Let $H=[a_{0},...,a_{h-1}]\in\cC_h,\, h\geq 1$
and either $\pi(H)=h$ or $h$ is not divisible by $\pi(H)$. Then $\pi(H^{n/h})=h$
for each $n\geq 2h$.
\end{lemma}
\begin{proof} For each $B\in\cC_m$ set
\[
\cO(B)=\{k\leq m,k\geq 1:\, B\cap T^{-k}B\ne\emptyset\},
\]
so that $\pi(B)=\min(\cO(B))$. It follows from Theorem 6 and Remark 7
from \cite{Sc} that we can also write
\begin{eqnarray}\label{3.16}
&\cO(B)=\{\pi(B),2\pi(B),...,[\frac m{\pi(B)}]\pi(B)\}\\
&\cup\{ k\in\{ m-\pi(B)+1,...,m\}:\, B\cap T^{-k}B\ne\emptyset\}.\nonumber
\end{eqnarray}
Let $n\geq 2h$ and set $\pi(H^{n/h})=k$. Clearly, $h\in\cO (H^{n/h})$ 
and so $h\geq k$. We suppose that $h>k$ and arrive at a contradiction. 
By the definition of $\pi$ it follows that $H^{n/h}\cap T^{-k}(H^{n/h})\ne
\emptyset$. Therefore, $[a_{k},...,a_{h-1}]=[a_{0},...,a_{h-1-k}]$ and 
$[a_{0},...,a_{k-1}]=[a_{h-k},...,a_{h-1}]$, and so
$k,h-k\in\mathcal{O}(H)$. Thus, $k\geq\pi(H)$ and if $\pi(H)=h$ we would
have $k\geq h$ which contradicts our assumption. Hence, it remains to consider
 the case when $h$ is not
divisible by $\pi(H)$. Since $k,h-k\in\mathcal{O}(H)$ then
$\pi(H)\leq k,h-k$, and so $k,h-k\leq h-\pi(H)$. This together with
 (\ref{3.16}) yields that there exist integers $a$ and $b$ such that
 $1\leq a,b\leq[\frac{h}{\pi(H)}]$, $k=a\cdot\pi(H)$ and $h-k=b\cdot\pi(H)$.
Therefore $h=(h-k)+k=(a+b)\pi(H)$ which contradicts our assumption that $h$ 
is not divisible by $\pi(H)$. Hence, $h=k$ and the proof is complete.
\end{proof}

The following result will be used in the proof of Theorem \ref{thm2.3}.
\begin{lemma}\label{lem3.6} Let $q_i,\, i=1,...,\ell$ be as in (\ref{2.10})
and $n\geq r(d_\ell+1)$.
For any positive integers $m$ and $n$ satisfying $m>2d_\ell n$ set
\[
\cN=\cN_{m,n}=\{l=1,2,..., n:\,\{X_m=1,X_{m+l}=1\}\ne\emptyset\}.
\]
Then $\min\cN=\ka$ where $\ka$ is defined by (\ref{2.11}) 
 with $n$ being the length of a cylinder $A$ there and $r=\pi(A)$.
\end{lemma} 
\begin{proof} It follows from (\ref{3.16}) that if $1\leq l\leq
\frac{n-r}{d_{\ell}}$ then $l\in\mathcal{N}$ if and only if $r$ divides
$d_il$ for each $i=1,.....,\ell$. By the definition $r$ divides $\ka d_i$
and by the assumption of the lemma $\ka\leq r\leq\frac{n-r}{d_{\ell}}$. Hence,
$\ka\in\mathcal{N}$, and so $\ka\geq y=\min\cN$. Now, let $l\in\mathcal{N}$
satisfies $l\leq\frac{n-r}{d_{\ell}}$.Then $r$ divides $d_il$
for each $1\leq i\leq\ell$, and so $\frac{r}{gcd\{r,d_{i}\}}$ divides $l$.
Thus, $\ka$ divides $l$, and so $\ka\leq l$. It follows that $\ka\leq y$, 
completing the proof of the lemma.
\end{proof}

In Sections \ref{sec5} and \ref{sec6} we will need the following results. 
As in Corollary \ref{cor2.2} for each $\om=(\om_0,\om_1,...)\in\Om$ set 
$A_n^\om=[\om_0,...,\om_{n-1}]$, $r^\om_n=\pi(A^\om_n)$ and   
$R^\om_n=[\om_0,...,\om_{r^\om_n-1}]$. Next, define $\ka_n^\om$ and
$\rho_n^\om$ by (\ref{2.11}) with $r=r_n^\om$ and $A=A_n^\om$.

\begin{lemma}\label{lem3.7} Given $\om\in\Omega$ the limit $r^{\om}=
\underset{n\rightarrow\infty}{\lim}r_{n}^{\om}$ exist. Furthermore, if $\om$
is a periodic point with period $d\in\mathbb{N}^{+}$ (i.e. the whole path 
 $\{T^{k}\om\::\:  k\geq 0\}$ of $\om$ consists of $d$ points) 
 then $r^{\om}=d$, otherwise $r^{\om}=\infty$. Assume that $\bbP$ is
 $\psi$-mixing. If $\om\in\Om_\bbP$ is not a periodic point then 
 \begin{equation}\label{3.17}
 \lim_{n\to\infty}\rho_n^\om=0.
 \end{equation}
 Furthermore, if independency conditions of Theorem \ref{thm2.7} are satisfied
then the limit $\rho^{\om}=\underset{n\rightarrow
\infty}{\lim}\rho_{n}^{\om}$ always exists.
Moreover, in this case if $\om$ is a periodic point with period 
$d\in\mathbb{N}^{+}$ then $\rho^{\om}=\big(\mathbb{P}([\om_{0},...,\om_{d-1}])
\big)^{k_0}$ with $k_0$ defined in Theorem \ref{thm2.7}(ii), otherwise 
$\rho^{\om}=0$.
\end{lemma}
\begin{proof} Assume that $\om$ is periodic with period $d\in\mathbb{N}^{+}$
and set $D=[\om_{0},...,\om_{d-1}]$. By the definition of a period of a point
it follows that $\pi(D)=d$ or $\pi(D)$ does not divide $d$. This together 
with Lemma \ref{lem3.5} yields that $\pi(D^{n/d})=d$
for all $n\geq2d$. Moreover, $T^{d}(\om)=\om$ which implies that
$D^{n/d}=A_{n}^{\om}$ for each $n\geq d$. From
this it follows that
\[
\underset{n\rightarrow\infty}{\lim}r_{n}^{\om}=\underset{n\rightarrow\infty}
{\lim}\pi(A_{n}^{\om})=\underset{n\rightarrow\infty}{\lim}\pi(D^{n/d})=d
\]
Now assume that $\om$ is not periodic. Given $n\geq1$, from the definition
of $\mathcal{O}$ in Lemma \ref{3.5} it follows that $r_{n+1}^{\om}\in
\mathcal{O}(A_{n}^{\om})\cup\{n+1\}$,
which implies that $r_{n}^{\om}\leq r_{n+1}^{\om}$. This holds true for all
$n\geq1$, and so $r^{\om}=\underset{n\rightarrow\infty}{\lim}r_{n}^{\om}$
exists and it is in the set $\mathbb{N}^{+}\cup\{\infty\}$. Assume by
contradiction that $r^{\om}<\infty$, then there exists an integer $M\geq1$
such that $r_{n}^{\om}=r^{\om}$ for all $n\geq M$. From this it follows that 
$A_{n}^{\om}=[\om_{0},...,\om_{r^{\om}-1}]^{n/r^{\om}}$ for all such $n$ 
which implies that $\om=[\om_{0},...,\om_{r^{\om}-1}]^{\infty}$ , and so
$\om$ is a periodic point which is a contradiction to our assumption.
From this it follows that $r^{\om}=\infty$ and the assertion concerning
$r^\om$ is proved. 

Next, let $\om\in\Om_\bbP$ be not a periodic point. Let $r=r_n^\om$, 
$\ka=\ka_n^\om$,
$A=A_n^\om$ and $R=R_n^\om$. Since $r$ divides $d_i\ka$ for each $i=1,...,\ell$
we see by (\ref{2.1}), (\ref{2.2}) and (\ref{3.1}) that
\begin{equation}\label{3.18}
\bbP(R^{(n+d_i\ka)/r})\leq (1+\psi_0)\bbP(A)\bbP(T^{n_r}R^{(n_r+r)/r})\leq 
(1+\psi_0)e^{-\Gam r^\om_n}\bbP(A)
\end{equation}
where $n_r=n$ (mod $r)=n-[n/r]r$.
Now (\ref{3.17}) follows since by the above $r_n^\om\to\infty$ as $n\to\infty$
when $\om$ is not periodic. 
Under the conditions of Theorem \ref{thm2.7} the remaining assertion
concerning $\rho^\om$ follows from the properties of $r^\om_n$ derived 
above together with the independence assumption.
\end{proof}

The assertion (\ref{2.12}) of Theorem \ref{thm2.3} we obtain as a separate 
lemma.
\begin{lemma}\label{lem3.8} Assume that $\bbP$ is $\psi$-mixing then
 (\ref{2.12}) holds true.
\end{lemma}
\begin{proof} First, clearly $\rho_A\leq 1$ for any $A\in\cC_n$ and 
each $n\geq 1$. Next, let
$A_{n}=[a_0^{(n)},...,a_{n-1}^{(n)}]\in\cC_{n}$ and write
\begin{equation}\label{3.19}
\bbP\big( R_{n}^{(n+d_1\ka_{n})/r_{n}}|A_{n}\big)=\frac {\bbP\big( 
R_{n}^{(n+d_1\ka_{n})/r_{n}}
\big)}{\bbP(A_n)}=1-\del_n,\,\,\del_n\geq 0 
\end{equation}
where $R_n=[a_0^{(n)},...,a_{r_n-1}^{(n)}]$,
$r_n=\pi(A_n)$ and $\ka_n$ is given by (\ref{2.11}) with $r=r_n$.
Then we show by induction that for any $k\in \bbN$,
\begin{equation}\label{3.20}
\bbP\big( R_n^{(n+kd_1\ka_n)/r_n}\big)\geq (1-2^{k-1}\del_n)\bbP(A_n). 
\end{equation}
Indeed, (\ref{3.20}) is satisfied for $k=1$ in view of (\ref{3.19}).
Suppose that (\ref{3.20}) holds true for $k=m$ and prove it for 
$k=m+1$. Set $n_r=n-r_n[n/r_n]$ and observe that for any $l\geq 1$,
\begin{eqnarray}\label{3.21}
& \big(R_n^{(n+(l-1)d_1\ka_n)/r_n}\cap 
T^{-(n+(l-1)d_1\ka_n)}(\Om\setminus T^{n_r}R_n^{(n_r+d_1\ka_n)/r_n})\big)\\
&\cup R_n^{(n+ld_1\ka_n)/r_n}=R_n^{(n+(l-1)d_1\ka_n)/r_n}\subset A_n\nonumber
\end{eqnarray}
and the union above is disjoint. The induction hypothesis together with
(\ref{3.21}) considered with $l=m$ yields
\begin{eqnarray}\label{3.22}
&\bbP\big(R_n^{(n+md_1\ka_n)/r_n}\cap T^{-(n+md_1\ka_n)}
(\Om\setminus T^{n_r}R_n^{(n_r+d_1\ka_n)/r_n})\big)\\
&\leq \bbP\big(T^{-d_1\ka}(R_n^{(n+(m-1)d_1\ka_n)/r_n}\cap T^{-(n+(m-1)d_1\ka_n)}
(\Om\setminus T^{n_r}R_n^{(n_r+d_1\ka_n)/r_n}))\big)\nonumber\\
&=\bbP\big(R_n^{(n+(m-1)d_1\ka_n)/r_n}\cap T^{-(n+(m-1)d_1\ka_n)}
(\Om\setminus T^{n_r}R_n^{(n_r+d_1\ka_n)/r_n})\big)\nonumber\\
&\leq\del_n2^{m-1}\bbP(A_n).
\nonumber \end{eqnarray}
Employing (\ref{3.21}) with $l=m+1$ we obtain from (\ref{3.22}) and
(\ref{3.20}) for $k=m$ that
\[
\bbP\big(R_n^{(n+(m+1)d_1\ka_n)/r_n}\big)\geq \bbP\big(R_n^{(n+md_1\ka_n)/r_n}
\big)-\del_n2^{m-1}\bbP(A_n)\geq (1-2^m\del_n)\bbP(A_n),
\]
and so (\ref{3.20}) holds true with $k=m+1$ completing the induction.

Now observe that by (\ref{2.1}), (\ref{2.2}) and (\ref{3.1}),
\begin{equation}\label{3.23}
\bbP\big( R_n^{(n+kd_1\ka_n)/r_n}\big)\leq (1+\psi_0)e^{-\Gam kr_n}
\bbP(A_n). 
\end{equation}
Since always $r_n\geq 1$ we can choose $k$ so large that 
$(1+\psi_0)e^{-\Gam kr_n}<\frac 12$ for all $n$ making the  right hand 
side of (\ref{3.23}) less than $\frac 12\bbP(A_n)$ for all $n$. Now
suppose by contradiction that there exists a subsequence $l_m\to\infty$
as $m\to\infty$ such that $\del_{l_m}\to 0$ as $m\to\infty$. Then, we
can choose $n=l_m$ in (\ref{3.20}) with $m$ so large that the right hand
side of (\ref{3.20}) will be bigger than $\frac 12\bbP(A_n)$ which leads
to a contradiction proving the lemma.
\end{proof}

\section{Poisson approximation}\label{sec4}\setcounter{equation}{0}

\subsection{Proof of Theorem \ref{2.1}} 
If $t^{-1}\gam(n)(\mathbb{P}(A))^{\ell}\geq\frac{1}{4}$
then the theorem clearly holds true, and so we can assume that 
$\gam(n)<\frac{t(\mathbb{P}(A))^{-\ell}}{4}$. Set $U=S_N^A$.
If $t(\mathbb{P}(A))^{-\ell}<4$ then for each $L\subset\mathbb{N}$,
\begin{eqnarray}\label{4.1}
&|\mathbb{P}\{U\in L\}-P_{t}(L)|\leq\mathbb{P}\{U\ne0\}+|\mathbb{P}\{U=0\}-
P_{t}\{0\}|+P_{t}(\mathbb{N}^{+})\\
&\leq2\mathbb{P}\{U\ne0\}+2P_{t}(\mathbb{N}^{+})\leq2N\mathbb{P}(A)
+2(1-e^{-t})\leq8\mathbb{P}(A)+2t\leq16\mathbb{P}(A).\nonumber
\end{eqnarray}
Again the theorem holds true, so we can assume that 
$t(\mathbb{P}(A))^{-\ell}\geq 4$
and then $\gam(n)<\frac{t(\mathbb{P}(A))^{-\ell}}{4}\leq N$.

Set $W=\underset{\alpha=\gam(n)}{\overset{N}{\sum}}X_{\alpha}$,
where $X_\al=X_\al^A$ was defined in (\ref{2.6}),
and $\lambda=EW$. For any $L\subset\mathbb{N}$,
\begin{eqnarray}\label{4.2}
|\mathbb{P}\{U\in L\}-P_{t}(L)|\leq|\mathbb{P}\{U\in L\}-\mathbb{P}\{W\in L\}|\\
+|\mathbb{P}\{W\in L\}-P_{\lambda}(L)|+|P_{\lambda}(L)-P_{t}(L)|=\delta_{1}
+\delta_{2}+\delta_{3}\nonumber
\end{eqnarray}
where $\del_1,\del_2$ and $\del_3$ denote the first, the second and the third
terms in the right hand side of (\ref{4.2}), respectively.
We estimate $\delta_{1}$ by
\begin{equation}\label{4.3}
\delta_{1}\leq2\mathbb{P}\{U-W>0\}\leq2\underset{\alpha=1}
{\overset{\gam(n)}{\sum}}\mathbb{P}\{X_{\alpha}=1\}\leq2\gam(n)
\mathbb{P}(A).
\end{equation}

In order to estimate $\delta_{2}$ we use Theorem 1 from \cite{AGG}. Note that
from the assumption $\psi_{n}\leq (3/2)^{1/(\ell+1)}-1$ and from Lemma 
\ref{lem3.2} whenever $\gam(n)\leq\alpha\leq N$ it follows that 
\[
\mathbb{P}\{X_{\alpha}=1\}\geq(2-(1+\psi_n)^\ell)(\mathbb{P}(A))^{\ell}>
\frac{1}{2}(\mathbb{P}(A))^{\ell}>0.
\]
Hence, the conditions of Theorem 1 from \cite{AGG} are satisfied with the 
collection $\{X_{\gam(n)},...,X_{N}\}$. For each $\al\in\bbN_+$
satisfying $\gam(n)\leq\alpha\leq N$ set
\[
B_{\alpha}=\{\be\geq\gam(n),\,\beta\leq N:\, \text{ such that }|q_{i}
(\alpha)-q_{j}(\beta)|<2n\,\,\mbox{for some}\,\, i,j=1,2,...,\ell\}.
\]
Then by Theorem 1 from \cite{AGG},
\begin{equation}\label{4.4}
\delta_{2}\leq b_{1}+b_{2}+b_{3}
\end{equation}
where 
\[
b_{1}=\underset{\alpha=\gam(n)}{\overset{N}{\sum}}(\underset{\beta\in 
B_{\alpha}}{\sum}\mathbb{P}\{X_{\alpha}=1\}\mathbb{P}\{X_{\beta}=1\}),
\]
\[
b_{2}=\underset{\alpha=\gam(n)}{\overset{N}{\sum}}(\underset{\alpha\ne
\beta\in B_{\alpha}}{\sum}\mathbb{P}\{X_{\alpha}=1,X_{\beta}=1\})\,\,
\mbox{and}\,\,
b_{3}=\underset{\alpha=\gam(n)}{\overset{N}{\sum}}E\big\vert E(X_{\alpha}
-EX_\alpha\mid\cB_\al)\big\vert
\]
with $\cB_\al=\sigma\{X_{\beta}\::\:\beta\notin B_{\alpha}\}$.

From Lemma \ref{lem3.2} and the fact that 
$|B_{\alpha}|\leq4\ell^{2}n$ for each $\alpha$ it follows that 
\begin{equation}\label{4.5}
b_{1}\leq N4\ell^{2}n\big((1+\psi_n)\mathbb{P}(A)\big)^{2\ell}
\leq 8\ell^{2}nt(\mathbb{P}(A))^{\ell}.
\end{equation}

Next, we estimate $b_{2}$. Let $\gam(n)\leq\alpha\leq N$ and $\alpha\ne\beta
\in B_{\alpha}$.
Assume without loss of generality that $\alpha<\beta$, so $q_{1}(\alpha)
<q_{1}(\beta)$.
If $q_{1}(\beta)-q_{1}(\alpha)<\pi(A)$ then by the definition of
$\pi(A)$ it follows that $\mathbb{P}\{X_{\alpha}=1,X_{\beta}=1\}=0$,
so we can assume that $q_{1}(\beta)-q_{1}(\alpha)\geq\pi(A)$. Now
by (\ref{2.1}), (\ref{2.2}) and Lemmas \ref{lem3.1} and \ref{3.2},
\begin{eqnarray*}
&\mathbb{P}\{X_{\alpha}=1,X_{\beta}=1\}\leq\mathbb{P}(T^{-q_{1}(\alpha)}
([a_{0},...,a_{\pi(A)-1}])\cap\{X_{\beta}=1\})\\
&\leq(1+\psi_{0})\mathbb{P}(A(\pi))\mathbb{P}\{X_{\be}=1\}
\leq (1+\psi_0)\mathbb{P}(A(\pi))(1+\psi_n)^\ell(\bbP(A))^\ell.
\end{eqnarray*}
Since by our assumption $(1+\psi_n)^\ell\leq 3/2$ we obtain that
\begin{equation}\label{4.6}
b_{2}\leq 6(1+\psi_0)N\ell^{2}n\mathbb{P}(A(\pi))(\mathbb{P}(A))^{\ell}
\leq 6(1+\psi_{0})\ell^2tn\mathbb{P}(A(\pi)).
\end{equation}

In order to estimate $b_{3}$ we use Lemma \ref{3.3}. Fix an integer
$\al$ such that $\gam(n)\leq\al\leq N$ and set 
\begin{eqnarray*}
&Q=Q_\al=\{ q_i(\al)+m:\, i=1,...,\ell;\, m=0,1,...,n-1\}\,\,\mbox{and}\\
&\tilde Q=\tilde Q_\al=\{ q_j(\be)+m:\, j=1,...,\ell;\,\,\be\not\in B_\al,\, 
m=0,1,...,n-1\}.
\end{eqnarray*}
Then the conditions of Lemma \ref{lem3.3} are satisfied with $d=n$ and
such $Q$ and $\tilde Q$. Taking into account that $\cB_\al\subset\cF_{\tilde Q}$ we derive
easily from Lemmas \ref{lem3.2} and \ref{lem3.3} that for $p=EX_{\alpha}$, 
\begin{eqnarray*}
&E\bigr|E(X_{\alpha}-p\mid\mathcal{B_\al})\bigr|=E\bigr|E(E(X_{\alpha}-
p\mid\mathcal{F}_{\tilde Q})\mid\mathcal{B_\al})\bigr|\\
&\leq EE(\bigr|E(X_{\alpha}
-p\mid\mathcal{F}_{\tilde Q})\bigr|\mid\mathcal{B_\al})\\
&=E\bigr|E(X_{\alpha}-p\mid\mathcal{F}_{\tilde Q})\bigr|
\leq 2^{2\ell+4}\psi_{n}\mathbb{P}\{X_{\alpha}=1\}\leq
2^{2\ell+5}\psi_{n}(\mathbb{P}(A))^{\ell}
\end{eqnarray*}
 Hence,
\begin{equation}\label{4.7}
b_{3}\leq N2^{2\ell+5}\psi_{n}(\mathbb{P}(A))^{\ell}\leq2^{2\ell+5} 
t\psi_{n}.
\end{equation}
In order to estimate $\delta_{3}$ we use Lemma \ref{lem3.4} which yields
\[
\delta_{3}\leq\underset{l=0}{\overset{\infty}{\sum}}|P_{\lambda}\{l\}-
P_{t}\{l\}|\leq2e^{|\lambda-t|}|\lambda -t|=2\wp(|\la-t|).
\]
We also have by Lemma \ref{lem3.2} that
\begin{eqnarray*}
&|\lambda-t|\leq|E(\underset{\alpha=\gam(n)}{\overset{N}{\sum}}X_{\alpha})-
N(\mathbb{P}(A))^{\ell}|+(\mathbb{P}(A))^{\ell}\\
&\leq\underset{\alpha=\gam(n)}{\overset{N}{\sum}}|\mathbb{P}\{X_{\alpha}=1\}
-(\mathbb{P}(A))^{\ell}|
+\gam(n)(\mathbb{P}(A))^{\ell}\\
&\leq N\psi_{n}2^{\ell}(\mathbb{P}(A))^{\ell}+\gam(n)(\mathbb{P}(A))^{\ell}
\leq 2^{\ell} t\psi_{n}+\gam(n)(\mathbb{P}(A))^{\ell}.
\end{eqnarray*}
It follows that
\begin{equation}\label{4.8}
\delta_{3}\leq2\wp(2^{\ell} t\psi_{n}+\gam(n)(\mathbb{P}(A))^{\ell}).
\end{equation}
Now (\ref{2.7}) follows from (\ref{4.1})--(\ref{4.8}) while (\ref{2.8})
follows from Lemma \ref{lem3.1},
completing the proof of the theorem. \qed

\subsection{Proof of Corollary \ref{cor2.2}}
 Set $c=3\Gamma^{-1}$ and fix $M\in\mathbb{N}$
such that $M>c\ln M$ and $\psi_{n}\leq (3/2)^{1/(\ell+1)}-1$ for all $n\geq M$.
Denote by $\Omega^{*}$ the set of all $\omega\in\Omega$ for which there
exist an $M(\omega)\geq M$ such that for each $n\geq M(\omega)$,
\[
\pi(A_{n}^{\omega})>n-c\ln n\;\text{ and }\;\mathbb{P}(A_{n}^{\omega})>0.
\]
Set $U^\om_n=S_N^{A_n^\om}$. Assuming (\ref{2.8}) it follows from
Theorem \ref{thm2.1} that for each $\omega\in\Omega^{*}$ and
$n\geq M(\omega)$,
\begin{eqnarray*}
&\underset{L\subset\mathbb{N}}{\sup}|\mathbb{P}\{U_{n}^{\omega}\in L\}-
P_{t}(L)|\\
&\leq 16e^{-\Gam n}\big(\ell^2nt+\gam(n)(1+t^{-1})+t\ell^2n^4(1+\psi_0)
\big)+2\wp(2^{\ell} t\psi_{n}+\gam(n)e^{-\Gam n})
\end{eqnarray*}
which gives (\ref{2.11}) and it remains to show that $\Omega^{*}$ has
the full measure. 

For each $n\geq M$ set
\[
B_{n}=\{\omega\::\:\pi(A_{n}^{\omega})\leq n-c\ln n\}.
\]
Fix $n\geq M$ and let $d=[n-c\ln n]$. For $a_0,a_1,...,a_d$ and $r\leq d$
set $A^a_r=[a_0,a_1,...,a_{r-1}]$ and $D^a_{r,n}=\{\om=(\om_0,\om_1,
...):\,\om_k=a_{k-r[k/r]}\,\,\forall\, k=r,r+1,...n-1\}$.
Then by (\ref{2.1}), (\ref{2.2}) and Lemma \ref{lem3.1},
\begin{eqnarray*}
&\mathbb{P}(B_{n})\leq\underset{r=1}{\overset{d}{\sum}}\mathbb{P}\{\omega:
A_{n}^{\omega}\cap T^{-r}(A_{n}^{\omega})\ne\emptyset\}
=\underset{r=1}{\overset{d}{\sum}}\big(\underset{a_{0},...,a_{r-1}
\in\mathcal{A}}{\sum}\mathbb{P}(A^a_r\cap D^a_{r,n})\big) \\
&\leq(1+\psi_{0})\underset{r=1}{\overset{d}{\sum}}\underset{a_{0},...,a_{r-1}
\in\mathcal{A}}{\sum}\mathbb{P}(A^a_r)\mathbb{P}(D^a_{r,n})
\leq (1+\psi_{0})\underset{r=1}{\overset{d}{\sum}} e^{-\Gamma(n-r)}\\
&\times\underset{a_{0},...,a_{r-1}\in\mathcal{A}}{\sum}\mathbb{P}(A^a_r)
=(1+\psi_0)\underset{r=1}{\overset{d}{\sum}} e^{-\Gamma(n-r)}
\leq d(1+\psi_{0}) e^{-\Gamma(n-d)}\\
&\leq n(1+\psi_{0}) e^{-\Gamma c\ln n}
=(1+\psi_{0}) n^{1-\Gamma c}=(1+\psi_{0}) n^{-2}.
\end{eqnarray*}
It follows that
\[
\underset{n=M}{\overset{\infty}{\sum}}\mathbb{P}(B_{n})\leq(1+\psi_{0})
\underset{n=M}{\overset{\infty}{\sum}}n^{-2}<\infty.
\]
Now from the Borel-Cantelli lemma we obtain that $\mathbb{P}\{B_{n}
\text{ i.o.}\}=0$ where i.o. stands for "infinitely often".
Set $D=\Om\setminus\Om_\bbP$ which, recall, is the union of cylinders $A$
with $\bbP(A)=0$. Since $\mathbb{P}(D)=0$ then
$\mathbb{P}(\Omega\setminus\Omega^{*})=\mathbb{P}(D\cup\{B_{n}\,\,\text{i.o.}\})
=0$ completing the proof of the corollary. \qed

\section{Compound Poisson approximation}\label{sec5}\setcounter{equation}{0}

\subsection{Proof of Theorem \ref{thm2.3}} 
First, recall that assertions conserning $\rho=\rho_A$ are contained in
Lemmas \ref{lem3.7} and \ref{lem3.8}. Throughout this subsection $A\in\cC_n$
will be fixed, and so we will write $X_k$ and $S_N$ for $X^A_k$ and $S_N^A$,
respectively. Next, set $K=5d_{\ell}rn$ and $\hat X_\al=X_\al$ if $K<\al\leq
N$ and $\hat X_\al=0$ if $\al\leq K$ or $\al> N$. Now define
\[
U=\sum_{\al=1}^N\hat X_\al=\sum_{\al=K+1}^NX_\al\,\,\mbox{and}\,\, 
X_{\alpha,j}=(1-\hat X_{\alpha-\ka})(1-\hat X_{\alpha+j\ka})
\prod_{k=0}^{j-1}\hat X_{\alpha+k\ka}.
\]
Observe that for any $m$,
\begin{equation}\label{5.1}
|S_N-U|\leq \sum_{\al=1}^KX_\al\,\,\mbox{and}\,\, |U-\sum_{\al=K+1}^N
\sum_{j=1}^mjX_{\al,j}|\leq m\sum_{\al=K+1}^N\prod_{k=0}^mX_{\al+k\ka}
\end{equation}
with $\ka$ defined by (\ref{2.11}). Introduce also
\[
I_{0}=\{K+1,...,N\}\times\{1,...,n_{0}\},
\]
where, recall, $n_{0}=[\frac{n}{r}]$,
\[ 
\lambda_{\alpha,j}=EX_{\alpha,j},\,\lambda=\underset{(\alpha,j)\in I_{0}}
{\sum}\lambda_{\alpha,j}\,\,\mbox{and}\,\,s=t(1-\rho)
\]
with $\rho$ defined in (\ref{2.11}).

Next, we estimate $|\lambda-s|$. For each $i=1,2,...,\ell$ set for brevity
$c_i=n+d_i\ka$. Then
\begin{eqnarray*}
&|\lambda-s|\leq(\mathbb{P}(A))^{\ell}+\bigr|\underset{\alpha=K+1}
{\overset{N}{\sum}}\underset{j=1}{\overset{n_{0}}{\sum}}\lambda_{\alpha,j}-
N(\mathbb{P}(A))^{\ell}(1-\rho)\bigr|\\
&\leq(K+1)(\mathbb{P}(A))^{\ell}+\underset{\alpha=K+1}{\overset{N}{\sum}}
\bigr|\underset{j=1}{\overset{n_{0}}{\sum}}\lambda_{\alpha,j}-
(\mathbb{P}(A))^{\ell}(1-\rho)\bigr|\\
&=(K+1)(\mathbb{P}(A))^{\ell}+\underset{\alpha=K+1}{\overset{N}{\sum}}\bigr|
\mathbb{P}(\underset{j=1}{\overset{n_{0}}{\cup}}\{X_{\alpha,j}=1\})-
(\mathbb{P}(A))^{\ell}+\underset{i=1}{\overset{\ell}{\prod}}\mathbb{P}
(R^{c_{i}/r})\bigr|\\
&\leq 2K e^{-\Gamma n}+\underset{\alpha=K+1}{\overset{N}{\sum}}
\mathbb{P}\{\underset{k=0}{\overset{n_{0}}{\prod}}X_{\alpha+k\ka}=1\}\\
&+\underset{\alpha=K+1}{\overset{N}{\sum}}\bigr|\mathbb{P}\{(1-X_{\alpha-\ka})
 X_{\alpha}=1\}-\mathbb{P}\{X_{\alpha}=1\}+\mathbb{P}(\underset{i=1}
{\overset{\ell}{\cap}}T^{-d_{i}(\alpha-\ka)}R^{c_i/r})
\bigr|\\
&+\underset{\alpha=K+1}{\overset{N}{\sum}}\big(\bigr|\mathbb{P}
\{X_{\alpha}=1\}-(\mathbb{P}(A))^{\ell}\bigr|+\bigr|\mathbb{P}(\underset{i=1}
{\overset{\ell}{\cap}}T^{-d_{i}(\alpha-\ka)}R^{c_i/r})-\underset{i=1}
{\overset{\ell}{\prod}}\mathbb{P}(R^{c_{i}/r})\bigr|\big)\\
&=2K e^{-\Gamma n}+\sigma_{1}+\sigma_{2}+\sigma_{3}.
\end{eqnarray*}
Here $\sigma_{1},\sigma_{2}\text{ and }\sigma_{3}$ denote the first,
second and third sums, respectively, and we use in the last inequality 
above that
\[
\underset{\alpha=K+1}{\overset{N}{\sum}}
\mathbb{P}\{\underset{k=0}{\overset{n_{0}}{\prod}}X_{\alpha+k\ka}=1\}
\geq \underset{\alpha=K+1}{\overset{N}{\sum}}\bigr|\mathbb{P}\{(1-X_{\alpha
-\ka}) X_{\alpha}=1\}-\mathbb{P}(\cup_{j=1}^{n_0}\{ X_{\alpha}=1\})\bigr|.
\]

In order to estimate $\sigma_1$ we observe that the choice of $K$
gives $\al(d_{i+1}-d_i)\geq 3d_in$ for any $\al>K$ and $i=1,...,\ell-1$.
It follows that whenever $0\leq m\leq n_0$ and $\al>K$ there exist disjoint
sets of integers $Q_1,Q_2,,...,Q_\ell$ satisfying (\ref{3.2}) with $k\geq 1$
and such that $T^{-d_i\al}\cap_{l=0}^{n_0}T^{-d_il\ka}A\in\cF_{Q_i}$, $i=1,
...,\ell$. Since $r$ divides $d_i\ka$ and $d_i\ka\leq d_ir<n$ by the assumption
then for such $m$ each $\cap_{l=0}^{m}T^{-d_il\ka}A$ is contained
in $D_{i,m}\cap T^{-d_im\ka}A$ where $D_{i,m}$ is a cylinder set of the length 
$d_im\ka\geq rm$ and such that 
$D_{i,m}\in\cF_{0,d_im\ka -1}$ while, clearly, $T^{-d_im\ka}A\in
\cF_{d_im\ka,\infty}$. Hence, relying on Lemmas \ref{lem3.1} and
 \ref{lem3.2} we conclude from here that
\begin{eqnarray}\label{5.2-}
&\sum_{\al=K+1}^N\bbP\{\prod_{l=0}^mX_{\al+l\ka}=1\}=\underset{\alpha=K+1}
{\overset{N}{\sum}}\mathbb{P}\big(\underset{i=1}{\overset{\ell}{\cap}}T^{-d_{i}
\alpha}\cap_{l=0}^{m}T^{-d_il\ka}A\big)\\
&\leq (1+\psi_0)^\ell\sum_{\al=K+1}^N\prod_{i=1}^\ell\bbP(\cap_{l=0}^{m}
T^{-d_il\ka}A)\nonumber \\
&\leq (1+\psi_0)^{2\ell}N(\bbP(A))^\ell e^{-\Gam\ell rm}
\leq (1+\psi_0)^{2\ell} te^{-\Gam\ell rm}.\nonumber
\end{eqnarray}
In particular, taking $m=n_0$ we obtain
\begin{equation}\label{5.2}
\sigma_{1}\leq(1+\psi_0)^{2\ell} te^{-\Gam\ell rn_0}\leq
(1+\psi_0)^{2\ell} te^{-\Gam n/2}.
\end{equation}

Next we show that the term $\sigma_{2}$ vanishes. Since $r$ divides
$d_i\ka$ for each $i=1,...,\ell$ we have that $R^{c_i/r}=R^{d_i\ka/r}
\cap(T^{-d_i\ka}A)$. It follows that for any $\al>K$,
\begin{eqnarray*}
&\underset{i=1}{\overset{\ell}{\cap}} T^{-d_i(\al-\ka)}R^{c_{i}/r}
=(\underset{i=1}{\overset{\ell}{\cap}}T^{-d_i(\al-\ka)}R^{d_{i}\ka/r})\cap
(\underset{i=1}{\overset{\ell}{\cap}}T^{-d_i\al}A)\\
&=(\underset{i=1}{\overset{\ell}{\cap}}T^{-d_i(\al-\ka)}A)
\cap(\underset{i=1}{\overset{\ell}{\cap}}T^{-d_i\al}A)=
\{X_{\alpha-\ka} X_{\alpha}=1\}.
\end{eqnarray*}
Hence,
\begin{eqnarray*}
&\mathbb{P}\{X_{\alpha}=1\}-\mathbb{P}(\underset{i=1}{\overset{m}{\cap}}
T^{-d_i(\al-\ka)}R^{c_{i}/r})=
\mathbb{P}\{X_{\alpha}=1\}-\mathbb{P}\{X_{\alpha-\ka} X_{\alpha}=1\}\\
&=\mathbb{P}\{(1-X_{\alpha-\ka})X_{\alpha}=1\},
\end{eqnarray*} 
and so $\sig_2=0$.

Next, we estimate $\sigma_{3}$ using Lemma \ref{lem3.2} similarly to above
which gives
\begin{equation}\label{5.3}
\sigma_{3}\leq 2^{\ell}N\psi_{n}(\mathbb{P}(A))^{\ell}\leq2^{\ell} t
\psi_{n}.
\end{equation}
Hence, by (\ref{5.2}) and (\ref{5.3}),
\begin{eqnarray}\label{5.4}
&|\lambda-s|\leq 2Ke^{-\Gamma n}+(1+\psi_{0})^{2\ell}
 t e^{-\frac{\Gamma}{2} n}+2^{\ell} t\psi_{n}\\
&\leq 2(1+\psi_{0})^{2\ell} K(t+1) e^{-\frac{\Gamma}{2} n}
+2^{\ell} t\psi_{n}.\nonumber
\end{eqnarray}

Next, assume that $\lambda=0$.  Then by (\ref{5.1}),
\begin{equation}\label{5.5}
\mathbb{P}\{S_N\ne 0\}\leq\sum_{\al=1}^K\bbP\{ X_\al=1\}+\sig_1\leq 
K\mathbb{P}(A)+\sigma_{1}.
\end{equation}
Let $\eta_{1},\eta_{2},...$ be a sequence of i.i.d. random variables with 
$\mathbb{P}\{\eta_{1}\in\mathbb{N}^{+}\}=1$ independent of a Poisson random
variable $W$ with the parameter $s$ and $Z=\sum_{k=1}^W\eta_k$. Then by 
(\ref{5.2}), (\ref{5.4}) and (\ref{5.5}) for any $L\subset\bbN$,
\begin{eqnarray}\label{5.6}
&\quad |\mathbb{P}\{S_N\in L\}-\mathbb{P}\{Z\in L\}|\leq\mathbb{P}\{S_N\ne0\}+
|\mathbb{P}\{S_N=0\}-\mathbb{P}\{Z=0\}|\\
&+\mathbb{P}\{Z\ne0\}\leq 2(\mathbb{P}\{S_N\ne0\}+\mathbb{P}\{Z\ne0\})=
2(\mathbb{P}\{S_N\ne0\}+(1-e^{-s}))\nonumber\\
&\leq2(\mathbb{P}\{S_N\ne0\}+s)\leq 8(1+\psi_{0})^{2\ell}
 K(t+1) e^{-\frac{\Gamma}{2} n}+2^{\ell+1} t\psi_{n}.\nonumber
\end{eqnarray}
Hence, if $\lambda=0$ the theorem follows for any such i.i.d. sequence 
$\eta_{1},\eta_{2},...$, and so we can assume that $\lambda>0$. 

Define 
\[
I=\{(\alpha,j)\in I_{0}\::\:\mathbb{P}\{X_{\alpha,j}=1\}>0\}
\]
When $\lambda>0$ then $I\ne\emptyset$. For each $j\in\{1,...,n_{0}\}$ set
\[
\lambda_{j}=\lambda^{-1}\underset{\alpha=K+1}{\overset{N}{\sum}}
\lambda_{\alpha,j}
\]
We choose an i.i.d. sequence $\{\eta_{k}\}_{k=1}^{\infty}$ such that 
$\mathbb{P}\{\eta_{1}=j\}=\lambda_{j}$ for each $1\leq j\leq n_{0}$
and set, again, $Z=\sum_{k=1}^W\eta_k$ where, as before, $W$ is
a Poisson random with the parameter $s$ independent of $\eta_k$'s.
Set $\mathbf{X}=\{X_{\alpha,j}\}_{(\alpha,j)\in I}$ and
let $\mathbf{Y}=\{Y_{\alpha,j}\}_{(\alpha,j)\in I}$ be a collection
of independent random variables such that each $Y_{\alpha,j}$
has the Poisson distribution with the parameter $\lambda_{\alpha,j}$.
Given $(a_{\alpha,j})_{(\alpha,j)\in I}=a\in\mathbb{N}^{I}$ define
\[
f(a)=\underset{(\alpha,j)\in I}{\sum}j a_{\alpha,j}\,\,\mbox{and}\,\,
h_L(a)=\bbI_{f(a)\in L}.
\]
Then
\begin{eqnarray}\label{5.7}
&|\mathbb{P}\{S_N\in L\}-\mathbb{P}\{Z\in L\}|\leq|\mathbb{P}\{S_N\in L\}-
Eh_L(\mathbf{X})|\\
&+|Eh_L(\mathbf{X})-Eh_L(\mathbf{Y})|+
|Eh_L(\mathbf{Y})-\mathbb{P}\{Z\in L\}|=\delta_{1}+\delta_{2}+
\delta_{3}\nonumber
\end{eqnarray}
where $\del_1,\del_2$ and $\del_3$ denote the respective terms in the right
hand side of (\ref{5.7}).

By (\ref{5.1}) and (\ref{5.2}) we obtain that
\begin{eqnarray}\label{5.8}
&\delta_{1}\leq 2\mathbb{P}\{S_N\ne f(\mathbf{X})\}\leq 2\sum_{\al=1}^K
\bbP\{ X_\al=1\}+2\sig_1\leq 2K\bbP(A)+2\sig_1\\
&\leq 2K\mathbb{P}(A)+2(1+\psi_{0})^{2\ell}t 
e^{-\frac{\Gamma}{2} n}\nonumber\leq 4(1+\psi_{0})^{2\ell}
K(t+1)e^{-\frac{\Gamma}{2} n}.
\nonumber\end{eqnarray}

In order to estimate $\delta_{2}$ we use Theorem 2 from \cite{AGG}. Note
that by the definition of $I$ for each $(\alpha,j)\in I$ we have
$\mathbb{P}\{X_{\alpha,j}=1\}>0$, and so the use of the theorem
is justified. For each $(\alpha,j)\in I$ define
\[
B_{\alpha,j}=\{(\beta,k)\in I\::\:\exists\: i_{1},i_{2}=1,...,\ell\,\,
\mbox{such that}\,\, |d_{i_{1}}\alpha-d_{i_{2}}\beta|<K\}.
\]
By Theorem 2 in \cite{AGG} we see that
\begin{equation}\label{5.9}
\delta_{2}\leq\bigl\Vert\cL(f(\mathbf{X}))-\cL(f(\mathbf{Y}))\bigr\Vert\leq
2(2b_{1}+2b_{2}+b_{3}),
\end{equation}
where 
\[
\bigl\Vert\cL(\xi)-\cL(\zeta)\bigr\Vert=2\sup_{L\subset\bbN}|\bbP\{\xi\in L\}
-\bbP\{\zeta\in L\}|
\]
is the total variation distance between distributions of nonnegative integer
valued random variables $\xi$ and $\zeta$, 
\begin{eqnarray*}
&b_{1}=\underset{(\alpha,j)\in I}{\sum}(\underset{(\beta,k)\in B_{\alpha,j}}
{\sum}\mathbb{P}\{X_{\alpha,j}=1\}\mathbb{P}\{X_{\beta,k}=1\}),\\
&b_{2}=\underset{(\alpha,j)\in I}{\sum}(\underset{(\alpha,j)\ne(\beta,k)\in 
B_{\alpha,j}}{\sum}\mathbb{P}\{X_{\alpha,j}=1,X_{\beta,k}=1\})\,\,\mbox{and}\\
&b_{3}=\underset{(\alpha,j)\in I}{\sum}E\bigr|E(X_{\alpha,j}-
\lambda_{\alpha,j}\mid\cB_{\al,j})\bigr|
\end{eqnarray*}
where $\cB_{\al,j}=\sigma\{X_{\beta,k}:\,(\beta,k)\notin B_{\alpha,j}\}$.
For each $(\alpha,j)\in I$ it follows by Lemma \ref{lem3.2} that
\[
\mathbb{P}\{X_{\alpha,j}=1\}\leq\mathbb{P}\{X_{\alpha}=1\}
\leq 2(\mathbb{P}(A))^{\ell}.
\]
Since the number of elements in $B_{\alpha,j}$  and in $I$ do not exceed
$2K\ell^{2}n$ and $nN$, respectively, we obtain from here that
\begin{equation}\label{5.10}
b_{1}\leq 8n^2NK\ell^{2}(\mathbb{P}(A))^{2\ell}\leq 8K\ell^{2} t n^{2}
(\mathbb{P}(A))^{\ell}\leq 8K\ell^{2} tn^{2} e^{-\Gamma n}.
\end{equation}

Now we estimate $b_{2}$. Fix $(\alpha,j)\in I$, let $(\alpha,j)\ne(\beta,k)\in 
B_{\alpha,j}$
and set $F=\{X_{\alpha,j}=1,X_{\beta,k}=1\}$.
We want to estimate $\mathbb{P}(F)$ from above. Clearly, if $\alpha=\beta$
then $F=\emptyset$, so without loss of generality we can assume that
$\alpha<\beta$. Suppose, first, that $\alpha+\frac{n-3r}{d_{\ell}}>\beta$
and show that in this case $F=\emptyset$. Indeed, assume by contradiction
that $F\ne\emptyset$ then by Lemma \ref{lem3.6} it follows that 
$\alpha\leq\beta-\ka$. Let $\om\in F$ then $X_{\beta,k}(\omega)=1$, and so
$X_{\beta-\ka}(\omega)=0$ and $X_{\beta}(\omega)=1$. Hence,
 there exists an $1\leq i_{0}\leq \ell$ such that
\[
\mathbb{I}_{A}\circ T^{d_{i_{0}}(\beta-\ka)}(\omega)=0\text{ and }
\mathbb{I}_{A}\circ T^{d_{i_{0}}\beta}(\omega)=1.
\]
It follows that for $c=d_{i_{0}} \ka$,
\begin{equation}\label{5.11}
\mathbb{I}_{R^{c/r}}\circ T^{d_{i_{0}}(\beta-\ka)}(\omega)=0.
\end{equation}
By our assumption,
\[
d_{i_{0}}(\beta-\alpha)< d_\ell(\be-\al)<n-3r<(n_{0}-2) r.
\]
Write $d_{i_{0}}(\beta-\alpha)=ur+v$ where $u,v\in\mathbb{N}$
and $v<r$. Then $d_{i_{0}}\alpha+ur\leq d_{i_{0}}\alpha+(n_{0}-2) r$.
Since $X_{\alpha,j}(\omega)=1$, and so $\bbI_A\circ T^{d_{i_0}\al}(\om)=1$,
we obtain from the last inequality and the definition of $n_0$ that
$\mathbb{I}_{R^{2}}\circ T^{d_{i_{0}}\alpha+ur}(\omega)=1$ where $R^2=
R^{2r/r}$ is the concatenation of two copies of $R$.
Since $X_{\beta,k}(\omega)=1$, and so $\bbI_A\circ T^{d_{i_0}\be}(\om)=1$,
we obtain also that 
\[
\mathbb{I}_{R^{2}}\circ T^{d_{i_{0}}\alpha+ur+v}(\omega)=
\mathbb{I}_{R^{2}}\circ T^{d_{i_{0}}\beta}(\omega)=1.
\]
Hence, $R^{2}\cap T^{-v}(R^{2})\ne\emptyset$. By the assumption $v<r$ and
if $v>0$ then $\pi(R)\leq\pi(R^2)\leq v<r$. If $r$ is not divisible by 
$\pi(R)$ then by Lemma \ref{lem3.5} we would have $\pi(R^2)=r$ contradicting
the above inequality and if $\pi(R)$ divides $r$ then $\pi(R)\in\mathcal O(A)$
contradicting $\pi(A)=r$. Hence, $v=0$, and so $d_{i_{0}}(\beta-\alpha)=ur$.
Since $X_{\alpha,j}(\omega)=1$, $r$ divides $d_{i_0}\ka$ and  $n\geq ur =
d_{i_{0}}(\beta-\alpha)\geq d_{i_0}\ka$ it follows from here that
\begin{eqnarray*}
&1=X_{\alpha}(\omega)\leq\mathbb{I}_{A}\circ T^{d_{i_{0}}\cdot\alpha}(\omega)
\leq\mathbb{I}_{R^{ur/r}}\circ T^{d_{i_{0}}\cdot\alpha}(\omega)\\
&=(\mathbb{I}_{R^{(ur-d_{i_{0}}\ka)/r}}\circ T^{d_{i_{0}}\cdot\alpha}(\omega))
\cdot(\mathbb{I}_{R^{c/r}}\circ T^{(d_{i_{0}}\cdot\alpha+ur-d_{i_{0}}\ka)}(\omega))
\\
&\leq\mathbb{I}_{R^{c/r}}\circ T^{(d_{i_{0}}\cdot\alpha+ur-d_{i_{0}}\ka)}
(\omega)=\mathbb{I}_{R^{c/r}}\circ T^{d_{i_{0}}\cdot(\beta-\ka)}(\omega)
\end{eqnarray*}
where, again, $c=d_{i_0}\ka$ and we set $R^{0/r}=\Om$. Hence,
$\mathbb{I}_{R^{c/r}}\circ T^{d_{i_{0}}(\beta-\ka)}(\omega)=1$
contradicting (\ref{5.11}), and so $F=\emptyset$ in this case. 

Thus, we can assume that $\alpha+\frac{n-3r}{d_{\ell}}\leq\beta$, and so
 $d_{\ell}\alpha+n\leq d_{\ell}\beta+3r$. Hence, by Lemma \ref{lem3.2},
\begin{eqnarray*}
&\mathbb{P}\{X_{\alpha,j}=1,X_{\beta,k}=1\}\leq\bbP\{ X_\al=1,\,\bbI_A\circ
T^{d_\ell\be}=1\}\\
&\leq\mathbb{P}\{X_{\alpha}=1,\mathbb{I}_{R^{(n-3r)/r}}\circ T^{d_{\ell}
\beta+3r}=1\}\\
&\leq(1+\psi_{0})\mathbb{P}\{X_{\alpha}=1\}\mathbb{P}(R^{(n-3r)/r})
\leq 4\psi_{0}(\mathbb{P}(A))^{\ell}e^{-\Gamma(n-3r)}.
\end{eqnarray*}
Since the number of elements in $B_{\alpha,j}$  and in $I$ do not exceed
$2K\ell^{2}n$ and $nN$, respectively, we obtain that
\begin{equation}\label{5.12}
b_{2}\leq 8n^2NK\ell^{2}\psi_{0}(\mathbb{P}(A))^{\ell}e^{-\Gamma(n-3r)}
\leq 8\ell^{2}\psi_{0} K t n^{2} e^{-\frac{\Gamma}{2} n}.
\end{equation}

In order to estimate $b_{3}$ we use Lemma \ref{lem3.3} with
\begin{eqnarray*}
&Q=Q_{\al,j}=\{ d_i(\al+l\ka)+m:\, i=1,...,\ell;\, l=-1,0,1,...,j,\,
m=0,1,...,n-1\}\\
&\mbox{and}\,\,\tilde Q=\tilde Q_{\al,j}=\{ d_i(\be+l\ka)+m:\,\be\not\in 
B_{\al,j},\\ 
&i=1,...,\ell,\, l=-1,0,1,...,n_0,\, m=0,1,...,n-1\}.
\end{eqnarray*}
Then by the choice of $K$ the conditions of Lemma \ref{lem3.3} are satisfied
with $d=n$ and such $Q$ and $\tilde Q$. Taking into account that 
$\mathcal{B}_{\al,j}\subset\cF_{\tilde Q}$ 
we obtain from (\ref{2.1}), (\ref{2.2}) and Lemma \ref{lem3.3} that
\begin{eqnarray*}
&E\bigl|E(X_{\alpha,j}-\lambda_{\alpha,j}\mid\mathcal{B}_{\al,j})\bigl|=
E\bigl|E(E(X_{\alpha,j}-\lambda_{\alpha,j}\mid\mathcal{F}_{\tilde Q})
\mid\mathcal{B}_{\al,j})\bigl|\\
&\leq E\bigl|E(X_{\alpha,j}-\lambda_{\alpha,j}\mid\mathcal{F}_{\tilde Q})\bigl|
\leq 2^{2\ell+4}\psi_{n}\mathbb{P}\{X_{\alpha,j}=1\}.
\end{eqnarray*}
For each $1\leq i\leq \ell$ set $c_{i}=n+\ka d_{i}(j-1)$. Then by 
Lemma \ref{lem3.2} we see that
\begin{eqnarray*}
&\mathbb{P}\{X_{\alpha,j}=1\}\leq\mathbb{P}\{\underset{i=1}{\overset{m}{\cap}}
\mathbb{I}_{R^{c_{i}/r}}\circ T^{d_{i}\alpha}\}\leq 2\underset{i=1}
{\overset{\ell}{\prod}}\mathbb{P}(R^{c_{i}/r})\\
&\leq 2(1+\psi_{0})\underset{i=1}{\overset{\ell}{\prod}}(\mathbb{P}(A)
 e^{-\Gamma \ka d_{i}(j-1)})\leq 2(1+\psi_{0})(\mathbb{P}(A))^{\ell}
e^{-\Gamma\ell r(j-1)}.
\end{eqnarray*}

It follows from the above estimates that
\begin{eqnarray}\label{5.13}
&b_{3}=\underset{(\alpha,j)\in I}{\sum}E\bigl|E(X_{\alpha,j}-
\lambda_{\alpha,j}\mid\cB_{\al,j}\})\bigl|
\leq\underset{\alpha=K+1}{\overset{N}{\sum}}\underset{j=1}{\overset{n_{0}}
{\sum}}2^{2\ell+4}\psi_{n}\mathbb{P}\{X_{\alpha,j}=1\}\\
&\leq 2^{2\ell+5}(1+\psi_{0})\psi_{n}N(\mathbb{P}(A))^{\ell}\underset{j=1}
{\overset{n_{0}}{\sum}}e^{-\Gamma(j-1)}\nonumber\\
&\leq 2^{2\ell+5}(1+\psi_{0})\psi_{n} N(\mathbb{P}(A))^{\ell}
\frac{1}{1-e^{-\Gamma}}\leq 2^{2\ell+5}(1+\psi_{0})t\psi_{n}(1-e^{-\Gam})^{-1}.
\nonumber\end{eqnarray}

Next, we estimate $\delta_{3}$. Given a random variable $\xi$ we denote
by $\varphi_{\xi}$ the characteristic function of $\xi$. Let $\xi$ be a
Poisson random variable with a parameter $\la$ 
independent of $\{\eta_{k}\}_{k=1}^{\infty}$ and set $\Psi=\underset{k=1}
{\overset{\xi}{\sum}}\eta_{k}$.
Then for each $s\in\mathbb{R}$,
\begin{equation*}
\varphi_{\Psi}(s)=\underset{l=0}{\overset{\infty}{\sum}}\mathbb{P}\{\xi=l\}
\underset{k=1}{\overset{l}{\prod}}\varphi_{\eta_{k}}(s)=e^{-\la}\underset{l=0}
{\overset{\infty}{\sum}}\frac{\lambda^{l}}{l!}(\varphi_{\eta_{1}}(s))^{l}
=\exp(\lambda\underset{j=1}{\overset{n_{0}}{\sum}}\lambda_{j}(e^{i js}-1))
\end{equation*}
and
\begin{equation*}
\varphi_{f(\mathbf{Y})}(s)=\underset{(\alpha,j)\in I}{\prod}
\varphi_{Y_{\alpha,j}}(j s)
=\exp(\underset{(\alpha,j)\in I}{\sum}\lambda_{\alpha,j}(e^{i js}-1))=
\exp(\lambda\underset{j=1}{\overset{n_{0}}{\sum}}\lambda_{j}(e^{i js}-1))
\end{equation*}
warning the reader that the last two formulas are the only places in this
paper where $i$ stands for $\sqrt {-1}$ and not for an integer.
It follows from here that $f(\mathbf{Y})$ and $\Psi$ have the same
distribution, and so by Lemmas \ref{lem3.2} and \ref{lem3.4},
\begin{eqnarray}\label{5.14}
&\quad\delta_{3}=|\mathbb{P}\{\Psi\in L\}-\mathbb{P}\{Z\in L\}|=\underset{l=0}
{\overset{\infty}{\sum}}|\mathbb{P}\{\xi=l\}-\mathbb{P}\{W=l\}|
\mathbb{P}\{\underset{k=1}{\overset{l}{\sum}}\eta_{k}\in L\}\\
&\leq\underset{l=0}{\overset{\infty}{\sum}}|\mathbb{P}\{\xi=l\}-
\mathbb{P}\{W=l\}|\leq 2e^{|\lambda-s|}|\lambda-s|=2\wp(|\la-s|).\nonumber
\end{eqnarray}
Finally, (\ref{2.14}) follows from (\ref{5.4}), (\ref{5.6})--(\ref{5.10}) and 
(\ref{5.12})--(\ref{5.14}) completing the proof of Theorem \ref{thm2.3}.   \qed

\subsection{Proof of Corollary \ref{cor2.3+}}
Let $\om\in\Om_\bbP$ be a nonperiodic sequence, $A=A_n^\om$ and $r=r^\om_n=
\pi(A^\om_n)$. Assume first that $n>r^\om_n(d_\ell+6)$ so that Theorem
 \ref{thm2.3} can be applied. Then by (\ref{5.2-}) 
and the definition of the sequence $\eta_1,\eta_2,...$ in the proof of
Theorem \ref{thm2.3} we obtain that
\begin{equation}\label{5.16}
\bbP\{\eta_1=j\}=\la_j=\la^{-1}\sum_{\al=K+1}^N\la_{\al,j}\leq\la^{-1}
(1+\psi_0)^{\ell+1}te^{-\Gam\ell r(j-1)}.
\end{equation}
Let $\Xi$ be a Poisson random variable with 
the parameter $t$ and $Z=\sum^W_{l=1}\eta_l$ be the compound Poisson random
variable constructed by Theorem \ref{thm2.3} with $A=A^\om_n,\,\rho=\rho_n^\om$
 and $r=r^\om_n$. Then for any $L\subset\bbN$,
 \begin{eqnarray}\label{5.17}
 &\,\,\quad|\bbP\{\Xi\in L\}-\bbP\{ Z\in L\}|\leq\sum_{l=0}^\infty|\bbP\{\Xi=l\}
 -\bbP\{ W=l\}|\bbP\{\sum_{k=1}^l\eta_k\in L\}\\
 &+\sum_{l=0}^\infty\bbP\{\Xi=l\}|\bbP\{\sum_{k=1}^l
 \eta_k\in L\}-\bbI_L(l)|\nonumber
 \end{eqnarray}
 where $\bbI_L(l)=1$ if $l\in L$ and $=0$, otherwise. 
 
 By Lemma \ref{lem3.4},
 \begin{equation}\label{5.18}
 \sum_{l=0}^\infty |\bbP\{\Xi=l\}-\bbP\{ W=l\}|\leq 2t\rho^\om_ne^{t\rho^\om_n}
 =2\wp(t\rho^\om_n).
 \end{equation}
 Next, by (\ref{5.16}),
 \begin{eqnarray}\label{5.19}
 &|\bbP\{\sum_{k=1}^l\eta_k\in L\}-\bbI_L(l)|\leq 2\bbP\{\sum_{k=1}^l
 \eta_k\ne l\}\leq 2l\bbP\{\eta_1\ne 1\}\\
 &=2l\sum^{n_0}_{j=2}\la_j\leq 2l\la^{-1}(1+\psi_0)^{2\ell}te^{-\Gam\ell
 r^\om_n}(1-e^{-\Gam\ell r^\om_n})^{-1}.\nonumber
 \end{eqnarray}
 Now (\ref{5.17})--(\ref{5.19}) yield
 \begin{equation}\label{5.20}
 |\bbP\{\Xi\in L\}-\bbP\{ Z\in L\}|\leq 2\wp(t\rho^\om_n)+2\la^{-1}
 (1+\psi_0)^{\ell+1}t^2e^{-\Gam\ell r^\om_n}(1-e^{-\Gam\ell r^\om_n})^{-1}
 \end{equation}
 where we used that $t=E\Xi=e^{-t}\sum_{l=0}^\infty l\frac {t^l}{l!}$.
 
 If $n\leq r^\om_n(d_\ell+6)$ then we apply Theorem \ref{thm2.1}, and
 so we can write
 \[
 |\bbP\{ S_N^{A_n^\om}\in L\}-\bbP\{\Xi\in L\}|\leq\max(\ve_1(n)+\ve_2(n),
 \,\ve_3(n))
 \]
 where $\ve_1(n)$ and $\ve_2(n)$ are right hand sides of (\ref{2.14})
 and (\ref{5.20}), respectively, while
 \[
 \ve_3(n)=16e^{-\Gam n/(d_\ell+6)}\big(\ell^2nt+\gam(n)(1+t^{-1})+
 tn\ell^2(1+\psi_0)\big)+2\wp\big(2^\ell t\psi_n+\gam(n)e^{-\Gam n}\big).
 \]
 Clearly, $\ve_1(n),\ve_3(n)\to 0$ as $n\to\infty$ and
 since $\rho^\om_n\to 0$ and $r^\om_n\to\infty$ as $n\to\infty$ by 
 Lemma \ref{lem3.7} then $\ve_2(n)\to 0$ as $n\to\infty$, as well, and
 so the assertion of Corollary \ref{cor2.3+} follows.   \qed

\subsection{Proof of Corollary \ref{cor2.5}}
 Let $t>0$ be given. If $n>r(d_{\ell}+6)$
then we take $W$ and $Z$ to be as in Theorem \ref{thm2.3}. Note that 
$\mathbb{P}\{(\mathbb{P}(A))^{\ell}\tau_{A}>t\}=\mathbb{P}\{S^A_N=0\}$
and $\mathbb{P}\{Z=0\}=\bbP\{ W=0\}$, and so by Theorem \ref{thm2.3},
\begin{eqnarray}\label{5.21}
&|\mathbb{P}\{(\mathbb{P}(A))^{\ell}\tau_{A}>t\}-\mathbb{P}\{ W=0\}|=
|\mathbb{P}\{ S^A_N=0\}-\mathbb{P}\{ Z=0\}|\\
&\leq 2^{2\ell+7}(1+\psi_0)^{2\ell}(t+1)\big(d_\ell\ell^2 n^{4} 
e^{-\Gamma n/2}+\psi_{n}(1-e^{-\Gam})^{-1}\big)\nonumber\\
&+2\wp\big(2^\ell t\psi_n+10e^{-\Gam n/2}(1+\psi_0)^{2^\ell}d_\ell 
n^2(t+1)\big)\nonumber
\end{eqnarray} 

On the other hand, if $n\leq r(d_{\ell}+6)$ then by Theorem
\ref{thm2.1} (with $q_{i}$'s being linear),
\begin{equation}\label{5.22}
|\mathbb{P}\{(\mathbb{P}(A))^{\ell}\tau_{A}>t\}-P_s\{ 0\}|\leq
|\mathbb{P}\{S^A_N=0\}-P_{t}\{0\}|+|P_{t}\{0\}-P_{s}\{0\}|.
\end{equation}
where $s=t(1-\rho)$. Furthermore, $\ka\geq\frac{r}{d_{1}}$, and so
\begin{eqnarray}\label{5.23}
&\rho=\underset{i=1}{\overset{\ell}{\prod}}\mathbb{P}\{R^{(n+d_{i} \ka)/r}
\mid A\}\leq\frac{\mathbb{P}(R^{(n+d_{1} \ka)/r})}{\mathbb{P}(A)}
\leq\psi_{0}\mathbb{P}(R^{(d_{1} \ka)/r})\\
&\leq\psi_{0} e^{-\Gamma d_{1}\ka}\leq\psi_{0} e^{-\Gamma r}.
\nonumber\end{eqnarray}
By Theorem \ref{thm2.1},
\begin{eqnarray}\label{5.24}
&|\mathbb{P}\{S_N^A=0\}-P_{t}\{0\}|\leq 16e^{-\Gam n}(\ell^2nt+
\gam(n)(1+t^{-1}))\\
&+6(1+\psi_0)tn\ell^2 e^{-\Gamma r}+2\wp\big(2^\ell t\psi_n+
\gam(n)e^{-\Gam n}\big)\nonumber
\end{eqnarray}
and by (\ref{5.23}),
\begin{equation}\label{5.25}
|P_{t}\{0\}-P_{s}\{0\}|\leq |t-s|\leq t\psi_0 e^{-\Gam r}.
\end{equation}
Taking into account that $\gam(n)\leq 2n$ when $q_i$'s are as in Theorem 
\ref{thm2.3} and that $r\geq n/(d_\ell+6)$ in (\ref{5.24}) we obtain the
estimate of Corollary \ref{cor2.5} from (\ref{5.21})--(\ref{5.25}). \qed

\subsection{Proof of Corollary \ref{cor2.5+}}

First, observe that our mixing conditions imply, in particular, that
$\bbP$ is ergodic. Set $\iota=\sup_{A\in\cC_n,\, n\geq 1}\rho_A$ and
$b(n)=\frac {2\ln n}{1-\iota}$ recalling that $\iota<1$ by (\ref{2.12}). 
 Now, by Corollary \ref{cor2.5} for any $A\in\cC_n$ with $\bbP(A)>0$,
 \begin{eqnarray}\label{5.27}
 &\bbP\{ (\bbP(A))^\ell\tau_A>b(n)\}\leq e^{-(1-\iota)b(n)}\\
 &+2^{2\ell+8}(1+\psi_0)^{2\ell}(b(n)+1)\big( d_\ell\ell^2 n^4(1+
 \frac 1{b(n)})e^{-\Gam n/(d_\ell+6)}+\psi_n(1-e^{-\Gam})^{-1}\big)\nonumber\\
 &+2\wp\big(2^\ell b(n)\psi_n+10e^{-\Gam n/2}(1+\psi_0)^{2\ell}d_\ell n^2
 (b(n)+1)\big)\nonumber
 \nonumber\end{eqnarray}
and
\begin{eqnarray}\label{5.28}
 &\bbP\{ (\bbP(A))^\ell\tau_A\leq n^{-2}\}=1-\bbP\{ (\bbP(A))^\ell\tau_A
 >n^{-2}\}\\
 &\leq |1-e^{-(1-\rho_A)n^{-2}}|+|\bbP\{ (\bbP(A))^\ell\tau_A>n^{-2}\}-
 e^{-(1-\rho_A)n^{-2}}|\nonumber\\
 &\leq n^{-2}+2^{2\ell+8}(1+\psi_0)^{2\ell}(n^{-2}+1)\big( d_\ell\ell^2 n^4
 (1+n^2)e^{-\Gam n/(d_\ell+6)}\nonumber\\
 &+\psi_n(1-e^{-\Gam})^{-1}\big)+2\wp\big(2^\ell n^{-2}\psi_n+10
 e^{-\Gam n/2}(1+\psi_0)^{2\ell}d_\ell(1+n^2)\big)\nonumber
 \nonumber\end{eqnarray}.
 Applying (\ref{5.27}) and (\ref{5.28}) to $A=A^\om_n$ with $\om\in\Om_\bbP$
 we obtain by the Borel-Cantelli lemma
 that there exists a random variable $m=m(\varpi)$ on $\Om$ finite $\bbP$-a.s.
 and such that for all $n\geq m(\varpi)$, 
 \begin{equation}\label{5.29}
 n^{-2}<(\bbP(A_n^\om))^\ell\tau_{A^\om_n}(\varpi)\leq b(n)
 \end{equation}
 which implies (\ref{2.102}). Finally, if (\ref{2.103}) holds true then
 (\ref{2.104}) follows from (\ref{2.102}) and the Shannon-McMillan-Breiman
 theorem (see, for instance, \cite{Pe}).
 \qed

\section{I.i.d. case}\label{sec6}\setcounter{equation}{0}

\subsection{Proof of Theorem \ref{thm2.6}} 
We use the same notation as in the proof of Theorem \ref{thm2.3} and
as there we can assume that $\lambda>0$ and $N>K=5d_\ell rn$. In order to
derive (\ref{2.17}) we will estimate $|\mathbb{P}\{Z\in L\}-\mathbb{P}
\{Y\in L\}|$ for any $L\subset\mathbb{N}$ and combine it with (\ref{2.14}). 
Observe that under the assumption that the coordinate projections from $\Omega$
onto $\mathcal{A}$ are i.i.d. random variables it follows that 
\begin{equation}\label{6.2}
\rho=\underset{i=1}{\overset{\ell}{\prod}}\mathbb{P}\{R^{(n+d_{i} \ka)/r}
\mid A\}=\underset{i=1}{\overset{\ell}{\prod}}\mathbb{P}(R^{(d_{i} \ka)/r})
=\big(\bbP(R))^{k_0}\leq\mathbb{P}(R),
\end{equation}
where $k_0=\frac \ka r\sum_{i=1}^\ell d_i$, and also that 
$\lambda_{\alpha,j}=(1-\rho)^{2}(\mathbb{P}(A))^{\ell}\rho^{j-1}$ 
for each $(\alpha,j)\in I$. Let $\eta_1,\eta_2,...$ be
i.i.d. random variables constructed in the proof of Theorem \ref{thm2.3}.
Then, for each $ j=1,2,...,n_{0}=[n/r]$,
\begin{eqnarray*}
&\mathbb{P}\{\eta_{1}=j\}=\lambda_{j}=((N-K)\underset{l=1}
{\overset{n_{0}}{\sum}}\lambda_{N,l})^{-1}(N-K)\lambda_{N,j}\\
&=((1-\rho)^{2}(\mathbb{P}(A))^{\ell}\underset{l=1}{\overset{n_{0}}{\sum}}
\rho^{l-1})^{-1}(1-\rho)^{2}(\mathbb{P}(A))^{\ell}\rho^{j-1}\\
&=((1-\rho)\underset{l=1}{\overset{n_{0}}{\sum}}\rho^{l-1})^{-1} 
\mathbb{P}\{\zeta_{1}=j\}=\mathbb{P}\{\zeta_{1}=j\mid \zeta_{1}\in
\{1,...,n_{0}\}\}
\end{eqnarray*}
where $\zeta_1,\zeta_2,...$ are i.i.d. random variables described in the
statement of Theorem \ref{thm2.6}.
Furthermore, for any $j>n_{0}$,
\[
\mathbb{P}\{\eta_{1}=j\}=0=\mathbb{P}\{\zeta_{1}=j\mid \zeta_{1}\in\{1,...,
n_{0}\}\}
\]
Hence, by the independence for any $l\geq1$ and each
$j_{1},...,j_{l}\in\mathbb{N}^{+}$,
\[
\mathbb{P}\{\eta_{1}=j_{1},...,\eta_{l}=j_{l}\}=\mathbb{P}\{\zeta_{1}=j_{1},...,
\zeta_{l}=j_{l}\mid \zeta_{1},...,\zeta_{l}\in\{1,...,n_{0}\}\}.
\]
It follows from here that
\[
\mathbb{P}\{\underset{k=1}{\overset{l}{\sum}}\eta_{k}\in L\}=\mathbb{P}
\{\underset{k=1}{\overset{l}{\sum}}\zeta_{k}\in L\mid \zeta_{1},...,\zeta_{l}\in\{1,...,
n_{0}\}\}.
\]
Introduce the event $\Psi=\{\exists k\leq l,k\geq 1:\, \zeta_{k}\notin
\{1,...,n_{0}\}\}$. Then by the law of total probability it follows that
\begin{eqnarray*}
&|\mathbb{P}\{\underset{k=1}{\overset{l}{\sum}}\zeta_{k}\in L\}-\mathbb{P}
\{\underset{k=1}{\overset{l}{\sum}}\eta_{k}\in L\}|\\
&=\big\vert\mathbb{P}\{\zeta_{1},...,\zeta_{l}\in\{1,...,n_{0}\}\}\mathbb{P}
\{\underset{k=1}{\overset{l}{\sum}}\zeta_{k}\in L\mid \zeta_{1},...,\zeta_{l}\in
\{1,...,n_{0}\}\}\\
&+\mathbb{P}(\Psi)\mathbb{P}\{\underset{k=1}{\overset{l}{\sum}}\zeta_{k}\in L\mid
\Psi\}-\mathbb{P}\{\underset{k=1}{\overset{l}{\sum}}\zeta_{k}\in L\mid \zeta_{1},...
,\zeta_{l}\in\{1,...,n_{0}\}\}\big\vert\\
&=2\mathbb{P}(\Psi)\leq2\underset{k=1}{\overset{l}{\sum}}\mathbb{P}\{ \zeta_{k}
\notin\{1,...,n_{0}\}=2l\rho^{n_{0}}.
\end{eqnarray*}
Set again $s=t(1-\rho)$. Then
\begin{eqnarray*}
&|\mathbb{P}\{Z\in L\}-\mathbb{P}\{Y\in L\}|\leq\underset{l=0}{\overset{\infty}
{\sum}}\mathbb{P}\{W=l\}|\mathbb{P}\{\underset{k=1}{\overset{l}{\sum}}
\eta_{k}\in L\}-\mathbb{P}\{\underset{k=1}{\overset{l}{\sum}}\zeta_{k}\in L\}|\\
&\leq\underset{l=1}{\overset{n_{0}}{\sum}}|\mathbb{P}\{\underset{k=1}
{\overset{l}{\sum}}\eta_{k}\in L\}-\mathbb{P}\{\underset{k=1}{\overset{l}
{\sum}}\zeta_{k}\in L\}|+\mathbb{P}\{W>n_{0}\}\\
&\leq\underset{l=1}{\overset{n_{0}}{\sum}}2l\rho^{n_{0}}+
e^{-s}(e^{s}-\underset{l=0}{\overset{n_{0}}{\sum}}\frac{s^{l}}
{l!})\leq2n_{0}^{2}\rho^{n_{0}}+\frac{t^{n_{0}+1}}{(n_{0}+1)!}.
\end{eqnarray*}
Taking into account that $\psi_n\equiv 0$ under the independency assumption
 (\ref{2.17}) follows from here together with (\ref{2.14}), (\ref{6.2})
and Lemma \ref{lem3.1}, completing the proof of Theorem \ref{thm2.6}.  \qed

\subsection{Proof of Theorem \ref{thm2.7}}
In the i.i.d. case the assertion (i) can be derived easily studying the 
asymptotic behavior of the compound Poisson distribution constructed Theorem
\ref{thm2.6} but since (i) follows from the more general result of Corollary
\ref{cor2.3+} we will prove now only the assertion (ii).
 Assume that $\om$ is periodic with a period
$d\in\mathbb{N}^{+}$. From Lemma \ref{lem3.7} it follows that there
exist a positive integer $M\geq1$ such that $r_{n}^{\om}=d$ for all 
$n\geq M$. Next, by (\ref{6.2}) for any $n\geq M$,
\[
\rho_{n}^{\om}=\big(\mathbb{P}([\om_{0},...,\om_{r_{n}^{\om}-1}])\big)^{k_0}
=\big(\mathbb{P}([\om_{0},...,\om_{d-1}])\big)^{k_0}=\rho^{\om}
\]
which does not depend on $n$. Hence, for all $n\geq\max(M,d(d_\ell+6))$
we can apply Theorem \ref{thm2.6} in order to obtain that
\begin{eqnarray*}
&\underset{L\subset\mathbb{N}}{\sup}|\mathbb{P}\{U_{n}^{\om}\in L\}-
\mathbb{P}\{Y\in L\}|\\
&\leq 2^{2\ell+8}(t+1)d_\ell\ell^2 n^{4}e^{-\Gam n/2}+2\wp\big(10d_\ell 
n^2(t+1)e^{-\Gam n/2}\big)+\frac{t^{[n/d]+1}}{([n/d]+1)!}
\end{eqnarray*}
where $\Gamma=\min\{-\ln(\mathbb{P}\{\omega_{0}=a\})\::\: 
a\in\mathcal{A}\}$. 
The expression on the right hand side
of the last inequality tends to $0$ as $n\rightarrow\infty$, and so
(ii) is proved. 
\qed

\section{A nonconvergence example}\label{sec7}\setcounter{equation}{0}

We assume here that $\cA=\{ 0,1\}$ and consider the probability space
$(\Om,\cF,\bbP)$ such that $\Om=\{ 0,1\}^\bbN$, $\cF$ is the $\sig$-algebra
generated by cylinder sets and $\mathbb{P}$ is a probability
measure on $\Omega$ such that the coordinate projections 
$\{\omega_{j}\}_{j=0}^{\infty}$ from $\Omega$ onto $\{0,1\}$ are i.i.d. random
variables with
$\mathbb{P}\{\omega_{0}=0\}=p_{0}=1-p_{1}=1-\mathbb{P}\{\omega_{0}=1\}$
where $p_{0},p_{1}>0$, $p_{0}+p_{1}=1$ and $p_{0}\ne p_{1}$. As before
$T$ will denote the left shift on $\Om$ and we introduce here another map
$S:\,\Om\to\Om$ acting by $(S\om)_n=\om_n+\om_{n+1}$ (mod 2) for any $n\geq 0$
and $\om=(\om_0,\om_1,...)$.

Set $\bbP_0=S\bbP$ which is a probability measure on $\Om$ defined by
$\bbP_0(U)=\bbP(S^{-1}U)$ for any $U\in\cF$. We claim that $\bbP_0$ is
$T$-invariant. Indeed, let $A=[a_0,a_1,...,a_{n-1}]$ be a cylinder set
then
\begin{equation}\label{7.1}
S^{-1}A=[0,\al_1,...,\al_n]\cup [1,\be_1,...,\be_n]
\end{equation}
where $\al_i+\be_i=1$ for all $i=1,...,n$ and $a_i=\al_i+\al_{i+1}$ (mod 2)
$=\be_i+\be_{i+1}$ (mod 2) for $i=0,1,...,n-1$ where $\al_0=0$ and $\be_0=1$.
Similarly,
\begin{equation}\label{7.2}
S^{-1}T^{-1}A=[0,0,\al_1,...,\al_n]\cup [1,0,\al_1,...,\al_n]\cup
[1,1,\be_1,...,\be_n]\cup [0,1,\be_1,...,\be_n].
\end{equation}
It follows from (\ref{7.1}) and (\ref{7.2}) that
\begin{equation}\label{7.3}
\bbP_0(T^{-1}A)=\bbP(S^{-1}T^{-1}A)=\bbP(S^{-1}A)=\bbP_0(A).
\end{equation}
Since (\ref{7.3}) holds true for all cylinder sets, it remains true for
their disjoint unions and since it is preserved under monotone limits we
obtain that (\ref{7.3}) is satisfied for any $A\in\cF$ which proves our claim.

Next, we will show that $\bbP_0$ has a $\psi$-mixing coefficient $\psi_l$
equal zero for any $l\geq 1$ while $\psi_0<\infty$. First, observe that 
$\psi_l$ can be defined using only cylinder sets saying that $\psi_l$ is the
infimum of constants $M\geq 0$ such that for each $n\geq 0$ and any
cylinder sets $A\in\cF_{0,n}$ and $B\in\cF_{n+l+1,\infty}$ (as defined at
the beginning of Section \ref{sec2}),
\begin{equation}\label{7.4}
|\bbP_0(A\cap B)-\bbP_0(A)\bbP_0(B)|\leq M\bbP_0(A)\bbP_0(B).
\end{equation}
Indeed, if (\ref{7.4}) holds true for such cylinder sets then it remains true
for their corresponding disjoint unions and it is preserved under monotone
limits which yields that (\ref{7.4}) is satisfied for any sets $A\in\cF_{0,n}$
and $B\in\cF_{n+l+1,\infty}$, proving the assertion. Now, if $A\in\cF_{0,n}$
and $B\in\cF_{n+l+1,\infty}$ are cylinder sets then analysing their preimages
under $S^{-1}$ similarly to (\ref{7.1}) we conclude that $S^{-1}A\in
\cF_{0,n+1}$ and $S^{-1}B\in\cF_{n+l+1,\infty}$. Hence, if $l\geq 1$ then
$S^{-1}A$ and $S^{-1}B$ are independent events with respect to the probability
$\bbP$, and so
\[
\bbP_0(A\cap B)-\bbP_0(A)\bbP_0(B)=\bbP(S^{-1}A\cap S^{-1}B)-\bbP(S^{-1}A)
\bbP(S^{-1}B)=0
\]
implying that $\psi_l=0$ in this case.

In order to estimate $\psi_0$ we observe that if $B=[b_{n},b_{n+1},...,
b_{n+m-1}]=\{\om=(\om_0,\om_1,...):\,\om_i=b_i$ when $n\leq i\leq n+m-1\}$ then
\[
S^{-1}B=[\gam_{n},\gam_{n+1},...,\gam_{n+m}]\cup [\del_{n},\del_{n+1},
...,\del_{n+m}]
\]
where $\gam_i+\del_i=1$ for $i=n,...,n+m$ and $\gam_i+\gam_{i+1}=
\del_i+\del_{i+1}=b_i$ for $i=n,...,n+m-1$. This together with (\ref{7.1})
yields that
\[
\bbP_0(A\cap B)=\bbP(S^{-1}A\cap S^{-1}B)=\prod_{i=0}^np_{\al_i}
\prod_{j=n+1}^{n+m}p_{\gam_j}+\prod_{i=0}^np_{\be_i}
\prod_{j=n+1}^{n+m}p_{\del_j},
\]
where $\al_0=0$, $\be_0=1$ and without loss of generality we assume that 
$\al_n=\gam_{n}$ and $\be_n=\del_{n}$. On the other hand,
\[
\bbP_0(A)=\bbP(S^{-1}A)=\prod_{i=0}^np_{\al_i}+\prod_{i=0}^np_{\be_i}\,\,
\mbox{and}
\]
\[
\bbP_0(B)=\bbP(S^{-1}B)=\prod_{j=n}^{n+m}p_{\gam_j}+
\prod_{j=n}^{n+m}p_{\del_j}.
\]
It follows that
\[
\frac {\bbP_0(A\cap B)}{\bbP_0(A)\bbP_0(B)}\leq 2(p^{-1}_{\gam_{n}}+
p_{\del_{n}}^{-1})=2(p^{-1}_0+p^{-1}_1).
\]
Hence, $\psi_0\leq 1+2(p_0^{-1}+p_1^{-1})<\infty$ as required.

Let $1^{\infty}=\om\in\Omega$ and $t>0$. For each $n\geq1$ let $A_{n}^{\om}$,
$N_{n}^{\om}$ and $U_{n}^{\om}=S^{A^\om_n}_{N^\om_n}$
be as defined in Section \ref{sec2}. We now
show that the sequence $\{U_{n}^{\om}\}_{n=1}^{\infty}$ does not converge
in distribution when we take $\mathbb{P}_{0}$ as the measure on $\Omega$.
For two strings $S_{1}$ and $S_{2}$ of 0 and 1 digits we denote by 
$S_{1}\cdot S_{2}$
the concatenation of $S_{1}$ with $S_{2}$, and for any integer $k\geq1$
we denote by $S_{1}^{k}$ the concatenation of $S_{1}$ with itself
$k$ times. For every even $n\geq 1$,
\[
\mathbb{P}_{0}(1^{n})=\mathbb{P}([[1,0]^{n/2}\cdot 1])+
\mathbb{P}([[0,1]^{n/2}\cdot 0])=p_{1}^{\frac{n}{2}+1}
 p_{0}^{\frac{n}{2}}+p_{0}^{\frac{n}{2}+1} p_{1}^{\frac{n}{2}}=
(p_{1} p_{0})^{\frac{n}{2}}
\]
And for every odd $n\geq1$,
\[
\mathbb{P}_{0}(1^{n})=\mathbb{P}([[1,0]^{(n+1)/2}])+\mathbb{P}
([[0,1]^{(n+1)/2}])=2(p_{1} p_{0})^{\frac{n+1}{2}}.
\]
From this it follows that for every even $n\geq1$,
\[
\rho_{A^\om_n}=\mathbb{P}_{0}(1^{n+1}\mid1^{n})=\frac{2(p_{1} 
p_{0})^{\frac{n+2}{2}}}{(p_{1}p_{0})^{\frac{n}{2}}}=2
p_{1} p_{0}
\]
while for every odd $n\geq1$,
\[
\rho_{A^\om_n}=\mathbb{P}_{0}(1^{n+1}\mid1^{n})=\frac{(p_{1} 
p_{0})^{\frac{n+1}{2}}}{2(p_{1} p_{0})^{\frac{n+1}{2}}}=\frac{1}{2}.
\]
Now for each $n\geq1$ set $\theta_{n}=\mathbb{P}_{0}\{U_{n}^{\om}=0\}$.
Then from Theorem \ref{thm2.3} it follows that
\[
\underset{n\rightarrow\infty}{\lim}\theta_{2n}=e^{-t(1-2p_{1}p_{0})}\,\,
\mbox{and}\,\,\underset{n\rightarrow\infty}{\lim}\theta_{2n+1}=e^{-\frac{t}{2}}.
\]
Since $p_0\ne p_1$ then $1-2p_{1}p_{0}\ne\frac{1}{2}$,
and so the sequence $\{U_{n}^{\om}\}_{n=1}^{\infty}$ does not converge
in distribution when we take $\mathbb{P}_{0}$ as the measure on $\Omega$.

\bibliography{matz_nonarticles,matz_articles}
\bibliographystyle{alpha}

\end{document}